\documentclass[a4paper]{article}
\usepackage{authblk}
\usepackage{fullpage} 
\usepackage{parskip} 
\usepackage{tikz} 
\usepackage[light,math]{iwona}
\usepackage[T1]{fontenc}
\usepackage{amsmath}
\usepackage{amsfonts}
\usepackage{hyperref}
\usepackage{bbm}
\usepackage{mathabx}
\usepackage{url}
\usepackage{amsthm}
\usepackage{graphicx}
\usepackage{listings}
\usepackage{float}
\usepackage{tikz-cd}
\usepackage{setspace} 
\usepackage{todonotes}
\usepackage{imakeidx}
\makeindex

\usepackage[most]{tcolorbox}

\newtcolorbox{mytextbox}[1][]{%
	sharp corners,
	enhanced,
	colback=white,
	height=3cm,
	attach title to upper,
	#1
}

\theoremstyle{plain}
\newtheorem{theorem}{Theorem}[section] 
\newtheorem{corollary}[theorem]{Corollary}
\newtheorem{lemma}[theorem]{Lemma}
\newtheorem{prop} [theorem]{Proposition}

\theoremstyle{definition}
\newtheorem{defn}[theorem]{Definition}
\newtheorem{eg}[theorem]{Example}
\newtheorem{rem}[theorem]{Remark}

\newcommand{\E}{\mathbb{E}}

\newcommand{\PP}{\mathbb{P}}

\newcommand{\R}{\mathbb{R}}

\newcommand{\T}{\mathbb{T}}

\newcommand{\Z}{\mathbb{Z}}
\newcommand{\N}{\mathbb{N}}

\newcommand{\Bil}{\mathcal{B}}

\newcommand{\Fil}{\mathcal{F}}

\newcommand{\Hil}{\mathcal{H}}

\newcommand{\Lil}{\mathcal{L}}
\newcommand{\Mil}{\mathcal{M}}
\newcommand{\Nil}{\mathcal{N}}

\newcommand{\brw}{\boldsymbol{\mathrm{w}}}

\newcommand{\One}{\mathbbm{1}}

\title{Mean-field and fluctuations for hub dynamics in heterogeneous random networks}
\author[1,2]{Zheng Bian}
\affil[1]{Instituto de Ciências Matemáticas e de Computação, Universidade de São Paulo, São Carlos 13566-590, Brazil}
\author[2,3,4]{Jeroen S.W. Lamb}
\affil[2]{Department of Mathematics, Imperial College London, London SW7 2AZ, United Kingdom}
\affil[3]{International Research Center for Neurointelligence, The University of Tokyo, Tokyo 113-0033, Japan}
\affil[4]{Centre for Applied Mathematics and Bioinformatics, Department of Mathematics and Natural Sciences, Gulf University for Science and Technology, Halwally 32093, Kuwait}
\author[1,2]{Tiago Pereira}
\affil[ ]{Email addresses: \href{mailto:zheng@bian-zheng.cn}{zheng@bian-zheng.cn} \href{mailto:jeroen.lamb@imperial.ac.uk}{jeroen.lamb@imperial.ac.uk}, \href{mailto:tiago@icmc.usp.br}{tiago@icmc.usp.br}.}

\date{\today}
\begin{document}
\maketitle

\textbf{Abstract.} We study a class of heterogeneous random networks, where the network degree distribution follows a power-law, and each node dynamics is a random dynamical system, interacting with neighboring nodes via a random coupling function. We characterize the hub behavior by the mean-field, subject to statistically controlled fluctuations. In particular, we prove that the fluctuations are small over exponentially long time scales and obtain Berry-Esseen estimates for the fluctuation statistics at any fixed time. Our results provide an explanation for several numerical observations, namely, a scaling relation between system size and frequency of large fluctuations, the system size induced desynchronization, and the Gaussian behavior of the fluctuations.




\tableofcontents

\vspace{1cm}
\textbf{Data availability:} We do not analyse or generate any datasets, because our work proceeds within a theoretical and mathematical approach.

\textbf{All authors have no conflicts of interest.}

\newpage
\section{Introduction}
Network systems are fruitful models for various naturally occurring and man-made systems ranging from neuroscience \cite{HubScience} and physics \cite{marvel2009invariant} via electrochemistry \cite{nijholt2022emergent}, to social sciences  \cite{RMP} to mention a few applications. In the case of homogeneous networks, where 
symmetries facilitate the analysis, many ergodic and statistical properties of the network system are known, in the context of coupled map lattices \cite{chazottes2005dynamics}, all-to-all coupled systems \cite{Fernandez2019,Bahsoun2023-anosov}, and in the all-to-all thermodynamic limit \cite{ST2021-LR-SCTO,ST2022-sync, Galatolo2022}. Much less is known about heterogeneous networks, another class of realistic models which generally lack symmetry and feature massively connected nodes, referred to as \textit{hubs}, coexisting with poorly connected nodes \cite{judd2013exactly}. 

Hubs arise persistently in large random heterogeneous networks \cite{Banerjee2021} and play an important role in network systems. In addition to regulating the information flow and providing resilience during attacks \cite{albert2002statistical}, hubs affect the collective dynamics of the network  \cite{HubScience}. In fact, hubs may lead to a hierarchical transition toward global synchronization when the isolated dynamics of each node is periodic \cite{gomez2007paths}. Hubs can induce the optimal collective response of the network to noise \cite{tonjes2021coherence}, an abrupt transition to collective motion \cite{vlasov2015explosive}. When the isolated dynamics is chaotic, hubs inhibit global synchronization \cite{PvST20} but can spark the onset of cluster synchronization \cite{corder2023emergence,montalban2024spark}.

Even when individual interactions are weak, the hub behavior can change due to the collective interaction with its neighbors. Understanding hub behavior is intricate because the network system is high dimensional. Nonetheless, numerical and experimental results suggest that over very long time scales hub dynamics can be well approximated by a low-dimensional system given by the mean-field \cite{pereira2010hub,baptista2012collective,Ric16}. 
When each node dynamics is an expanding map, recent work \cite{PvST20} has proved this dimensional reduction under some restrictive assumptions.
Our results address three main shortcomings of previous works:
\begin{itemize}
\item[] i) Resilience against local perturbation. Previous work required 
hyperbolicity for the mean-fields of \textit{all} nodes. This assumption seems unnecessary: since hubs interact with a large number of nodes, the failure of a few nodes should not change the overall hub dynamics. 

\item[] ii) Networks with power-law degree distribution. Previous results also required the network to feature a degree separation between hubs and low degree nodes. This dichotomy is not present in most networks, where massively connected nodes coexist with other hubs that are not so well connected, leaving no gap between hubs and low degree nodes. 

\item[] iii) Characterization of large fluctuations. Over given time-scales hubs admit the mean-field approximation up to predominantly small fluctuations. It remains an open problem to statistically characterize the rare occurrences of large fluctuations in terms of network characteristics, such as size and degree distribution. 
\end{itemize}


In this paper, we meet these challenges by exploiting the typicality of random trajectories, and thereby overcoming some major technical challenges arising in the gap between topological dynamics and ergodic theory. We also apply concentration inequalities to the Chung-Lu random power-law network to establish certain graph theoretic properties.  
In particular, we advance the state-of-the-art of ergodic theory for network dynamics and complex systems by  characterizing hub dynamics in random power-law networks, in terms of the mean-field subject to Gaussian-like fluctuations.
We exhibit examples of uniformly contracting node dynamics in Introduction and Section \ref{sec:eg_contractions} and comment on  expanding and nonuniformly contracting cases in Remarks \ref{rem:star_typicality}, \ref{rem:beyond_contractions}.



In the following, we set up the network system in subsection \ref{sec:network_RDS_setup} and showcase numerical observations on the star network in subsections \ref{mfs}--\ref{sec:gaussian_star} and on a power-law network in subsections \ref{sec:emergent_PLhub_dynamics}--\ref{sec:size_desync}. In subsections \ref{sec:intro_informal_star} and \ref{sec:intro_informa_PL}, we  discuss our main theorems, which formalize the numerical observations. Section \ref{sec:intro_star} is dedicated to the case of star network, which is sufficient to explain the essence of the general result without too many technical details. Section \ref{sec:intro_PL} generalizes to a random power-law graph. In later sections \ref{sec:reduction_star} and \ref{sec:reduction_loc_star_like}, we provide abstract versions of the results covering both cases.

\subsection{Network random dynamical system} \label{sec:network_RDS_setup}
	The \textit{network random dynamical system} is the datum $(G,f,h,\alpha)$, where $G$ is an undirected graph on $N\geq 2$ nodes, $f=\{f_{\omega}:\T\to\T\}$ is a family of random circle maps that compose the node dynamics, $h=\{h_{\omega}:\T\times\T\to\R\}$ is a collection of random coupling functions that describe the pairwise interaction between neighbor nodes in $G$, and $\alpha >0$ is the coupling strength. The graph $G$ is represented by the adjacency matrix $A=(A_{ij})_{i,j=1}^N$, where $A_{ij}=1$ if nodes $i,j$ are connected and 0 otherwise. The degree of node $i$ is $k_i:=  \sum_{j=1}^N A_{ij}$, and the largest degree $\Delta_0:= \max_{i=1,\cdots,N}  k_i$. 
		The state $x_i^{t+1}$ of node $i$ at time $t+1$
		is given by
		\begin{equation}\label{eq:network_dynamics}
			x_i^{t+1} = f_{\boldsymbol{\omega}_i^t}(x_i^t) + \frac{\alpha}{\Delta_0} \sum_{j=1}^{N} A_{ij} h_{\boldsymbol{\omega}_i^t} (x_i^t,x_j^t)\mod 1,~~~~i=1,\cdots,N.
		\end{equation}
In the above equation, $\alpha$ is rescaled by $\Delta_0$ so that the most massively connected node receives an order-one interaction. 

\begin{rem}[Notation of noise realization]
In the introduction, we use the simple font $\omega$ to index the family of circle maps $f_{\omega}$ and coupling functions $h_{\omega}$, and the boldface $\boldsymbol{\omega}=(\boldsymbol{\omega}_i^t)$ to represent a vector of noise realizations with time $t=0,1,\cdots$ and node $i=1,\cdots,N$ coordinates.
\end{rem}

\begin{eg}\label{eg:intro_star_dynamics}
	As an example of node dynamics $f$, consider a family of contractions
	$f_\omega:\T\to\T$ on the circle
	\begin{equation}\label{eq:f_omega}
		f_\omega(x):=\begin{cases}
			\frac{x}{2} +\frac{\omega}{4},& x\in[0,\frac{1}{2}]\\
			\frac{1-x}{2} +\frac{\omega}{4},& x\in[\frac{1}{2},1]
		\end{cases}, ~~~~\omega=0,1,2,3,
	\end{equation}
	and iterate the dynamics at each time step by choosing from $\{f_{\omega}:\omega=0,1,2,3\}$ randomly independently and identically with probability $1/4$ for each contraction. 
	\begin{figure}[h]
		\centering 
		\includegraphics[height=50mm]{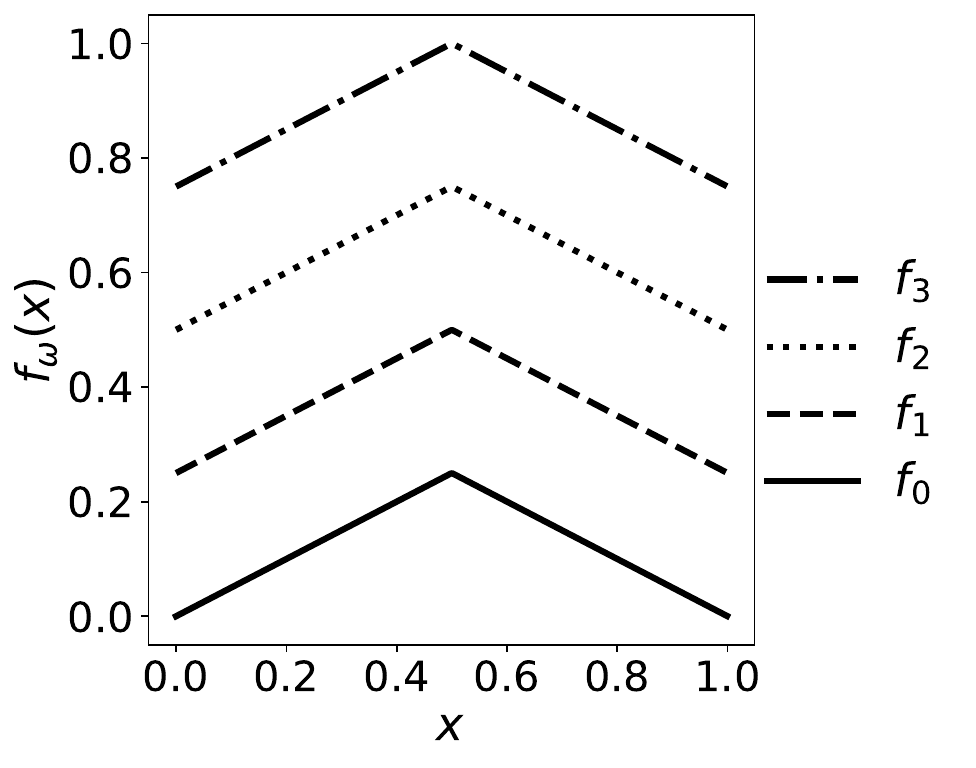}
		\caption{Graphs of four contractions $f_{\omega}$, $\omega=0,1,2,3$ on the circle $\T=[0,1]/0\sim1$.}
		\label{fig:f_omegas}
	\end{figure}
	In this example the node dynamics admit $\mathrm{Leb}_{\T}$ as the unique stationary measure. For other choices of $f_{\omega}$, the stationary measure can be absolutely continuous or singular with respect to $\mathrm{Leb}_{\T}$, including measures supported on a Cantor set; Theorem \ref{thm:typicality_randomcontractions} will cover a more general case.
	
	
	As an example of random coupling function, consider
	\begin{equation}\label{eq:h_omega}
		h_\omega (x,y)= \sin 2\pi y- \sin 2\pi x - \frac{\omega}{3.6},~~~~\omega=0,1,2,3.
	\end{equation}
	More generally, $h$ can be any family of $C^4$ maps $\T\times\T\to\R$.
\end{eg}
In the rest of the Introduction, we continue with the concrete examples of node dynamics $f_{\omega}$ and coupling function $h_{\omega}$ as in Example \ref{eg:intro_star_dynamics}.

\subsection{Dynamics on the star network}\label{sec:intro_star}
Consider a star graph $G$ on $N=10^6$ nodes: one hub together with $L=N-1$ low-degree nodes, where the hub influences each low-degree node, and each low-degree node influences the hub but not the other low-degree nodes. 
We index the hub by 1 and write $z$ for $x_1$, so the star network dynamics reads
\begin{equation}\label{eq:star_dynamics}
	\begin{aligned}
		z^{t+1} =& f_{\boldsymbol{\omega}_1^t}(z^t) + \frac{\alpha}{L} \sum_{j=2}^N h_{\boldsymbol{\omega}_1^t} (z^t,x_j^t)\mod1,\\
		x_j^{t+1} =& f_{\boldsymbol{\omega}_j^t}(x_j^t) + \frac{\alpha}{L} h_{\boldsymbol{\omega}_j^t}(x_j^t, z^t)\mod1,~~~~j=2,\cdots,N.
	\end{aligned}
\end{equation}
Here the noise $\boldsymbol{\omega}_i^t$ is assumed to be iid in time $t=0,1,\cdots$ and in node coordinates $i=1,\cdots,N$, assigning weight $1/4$ to each $\omega\in\{0,1,2,3\}$.

\subsubsection{Mean-field dimensional reduction} \label{mfs}
We simulate the star network by probing three coupling strengths $\alpha=0.05$, $0.8$, $0.9$. For each $\alpha$, we initialize the node states $(z^0,x_2^0,\cdots,x_N^0)\in \T^N$ at random with uniform distribution in $[0,1)$. Then, we iterate Eq. (\ref{eq:star_dynamics}), discard the first $5000$ iterates, and collect the next $1000$ iterates. In Figure \ref{fig:star_reduction}, we plot the hub return map
$z^t$ versus $z^{t+1}$ in red.  As the coupling strength $\alpha$ varies, the hub behavior differs from the isolated node dynamics. In Figure \ref{fig:f_omegas} on the left panel at $\alpha=0.05$, the hub remains contractive, on the central panel at $\alpha=0.8$, the hub dynamics appear to have an expanding region, and lastly, on the right panel at $\alpha=0.9$, the hub appears to hover around a fixed point near $0.2$. This shows the variety of hub behaviors emergent from the interactions, according to different coupling strengths. Due to  our choice of $h_{\omega}$ in Eq. (\ref{eq:h_omega}), the randomness in $f_{\alpha,\omega}$ collapses at $\alpha=0.9$, resulting in the single graph of $f_{0.9}=f_{0.9,\omega}$, $\omega=0,1,2,3$ as shown in the right panel of Figure \ref{fig:star_reduction}.

\begin{figure}[h]
	\centering 
	\includegraphics[height=50mm]{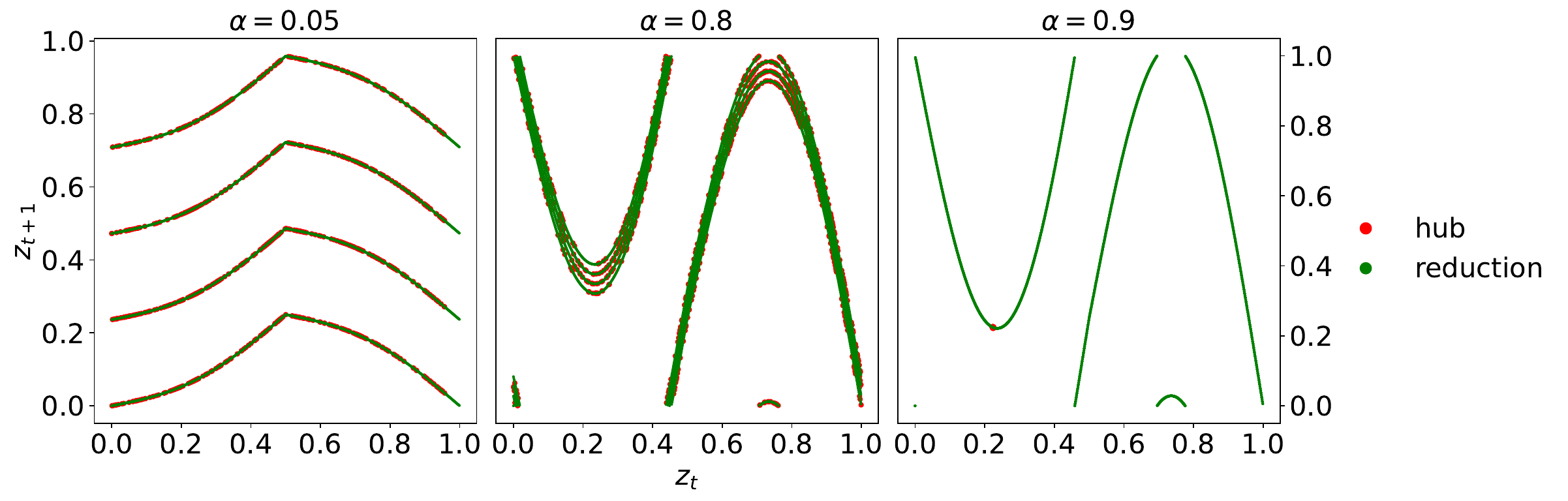}
	\caption{Numerical simulations for the star network dynamics (\ref{eq:star_dynamics}) on $N=10^6$  nodes at various coupling strengths $\alpha=0.05, 0.8,0.9$ on the left, center and right panels respectively, with iid random iteration of four circle contractions (\ref{eq:f_omega}) as isolated node dynamics and (\ref{eq:h_omega}) as pairwise interaction. The plots show the return behaviors of the hub, that is, the states $z^t$ on horizontal axis against the next states $z^{t+1}$ on vertical axis. Novel hub behaviors emerge from network interactions and vary across different coupling strengths: uniform contraction, expanding region and deterministic fixed point. The mean-field dimensional reduction ansatz yields a reduced one-dimensional system, whose graph, plotted in green, fits very well the actual hub behavior in red.}
	\label{fig:star_reduction}
\end{figure}

From Figure \ref{fig:star_reduction}, we observe that the emergent hub behavior resembles another one-dimensional random system.
Since the low degree nodes $j=2,\cdots,N$ receive only one contribution from the hub of order $O(L^{-1})$, 
we expect that the statistics of the low degree node to resemble the unique stationary measure $\mathrm{Leb}_{\T}$. In particular, the aggregate effect of the low degree nodes on the hub should be approximated by a space average 
\begin{align*}
	\frac{1}{L}\sum_{j=2}^N h_{\boldsymbol{\omega}_1^t}(z^t,x_j^t) &\approx \int_{\T} h_{\boldsymbol{\omega}_1^t}(z^t,x)\mathrm{d}x \\
	&= \sin 2\pi z^t - \frac{\boldsymbol{\omega}_1^t}{3.6}
\end{align*}
taken against the Lebesgue measure on the circle, which is the unique stationary measure of the isolated node dynamics and captures its typical statistics; see Theorem \ref{thm:typicality_randomcontractions}.  This ansatz, referred to as the \textit{mean-field dimensional reduction}, approximates the hub behavior by a one-dimension system
\[
z^{t+1} = f_{\alpha, \boldsymbol{\omega}_1^t}(z^t) + \xi^t(\boldsymbol{\omega},x),
\]
where the \textit{mean-field reduced map} reads
\begin{equation}\label{eq:f_alpha_omega}
	f_{\alpha, \omega}(z):= f_{\omega}(z) + \alpha\int_{\T} h_{\omega}(z, x)\mathrm{d}x\mod 1,
\end{equation}
and the \textit{mean-field fluctuation} at time $t$ from initial datum $(\omega,x)$ is
\[
\xi^t(\boldsymbol{\omega},x): = \frac{\alpha}{L} \sum_{j=2}^N \sin 2\pi x_j^t.
\]
Such approximation is meaningful when 
the fluctuation $| \xi^t | \ll 1$.
In Figure \ref{fig:star_reduction}, we plot in green the graph of this one-dimensional system, and the actual hub behavior in red, numerically corroborating the mean-field dimensional reduction. Theorem A (i) provides the corresponding mathematical statement, proving the reduction.

\subsubsection{Frequency of large fluctuations} \label{freqs}
To illustrate the impact of system size $L$ on fluctuations $\xi^t$, we simulate for each $L$ the star network dynamics at coupling strength $\alpha=0.9$ for $T$ iterations, and count the number 
\[
n^T_{\varepsilon}= \#\left\{t<T: |\xi^t|>\varepsilon\right\}
\] 
of times up to $T$ that the fluctuation exceeds a fixed threshold $\varepsilon$. Then we calculate the frequency $\rho_{\varepsilon}^T$ of large $(>\varepsilon)$ fluctuations up to time $T$
$$\rho_{\varepsilon}^T = n^T_{\varepsilon}/T.$$
In our simulations, we fix threshold $\varepsilon = 0.025$, vary the star size $L$ from $500$ to $10^4$ in steps of $500$, and simulate each system for a total time $T=2\times 10^5$.
In Figure \ref{fig:star_TT_size}, we show this frequency $\rho_{\varepsilon}^T$ versus $L$ in red diamonds. 
The green line is a linear fit of $\log \rho_{\varepsilon}^T$ against $L$, which strongly suggests that 
$$
\rho_{\varepsilon}^T ~= A e^{-\gamma L} ,~~~~A=e^{-0.736}, ~~\gamma=0.001
$$
and the chance to see the departure of the dynamics from mean-field reduced map becomes exponentially small in $L$ as the size $L$ of the star network grows. 
For quantitative relations between system size $L$  and frequency of large fluctuations $|\xi^t|>\varepsilon$, see Theorem A (ii) below.

\begin{figure}[h]
	\centering 
	\includegraphics[height=50mm]{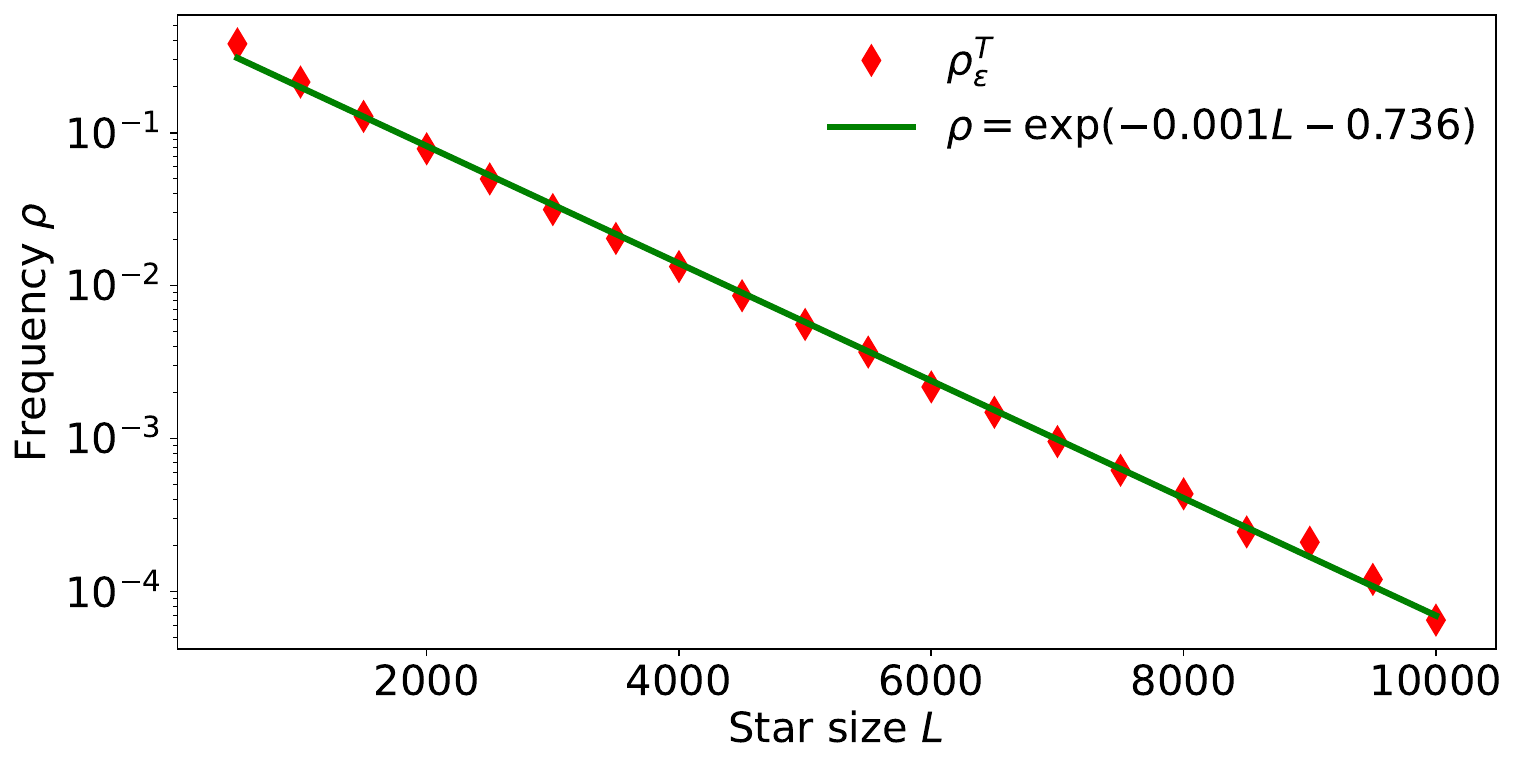}
	\caption{Frequency of large fluctuations decreases exponentially in system size. The red diamonds mark the frequency $\rho_{\varepsilon}^T$ up to time $T=2\times 10^5$ of large mean-field fluctuations, i.e., $|\xi^t|>\varepsilon$ with threshold $\varepsilon=0.025$. The horizontal axis for system size $L$ is in linear scale, whereas the vertical axis for frequency $\rho_{\varepsilon}^T$ is in logarithmic scale. The green line provides a tight linear fit, indicating an exponential decrease of $\rho_{\varepsilon}^T$ in $L$.}
	\label{fig:star_TT_size}
\end{figure}

\subsubsection{Gaussian behavior of the  fluctuations}\label{sec:gaussian_star}  
The fluctuation $\xi^t$ can be interpreted as an ensemble average of the low degree node states through the observable $x\mapsto \alpha\sin2\pi x$. The low degree nodes are almost isolated, up to hub influence of order $O(L^{-1})$, and hence almost independent from each other. Hence, we expect the fluctuations $\xi^t$ to follow a Central Limit behavior. 
To illustrate this, we fix star size $L=10^4$ and coupling strength $\alpha=0.9$, and take $10^4$ trials of network initial conditions randomly independently and uniformly in $\T^{L+1}$. For each trial $n$ we simulate the star network dynamics up to time $T=1000$ and calculate the terminal fluctuation $\xi^T_n$.
We plot the histogram of the data  $\{ \xi^T_n \}_{n=1}^{10^4}$ in Figure \ref{fig:star_CLT}. Superimposed in green is the probability density function of the normal distribution $\Nil(0,\frac{\alpha^2}{2L})$ with zero mean and variance $\frac{\alpha^2}{2L}$. The close fit indicates that $\Nil(0,\frac{\alpha^2}{2L})$ indeed captures the fluctuation statistics at time $T$; 
see Theorem A (iii).

\begin{figure}[h]
	\centering 
	\includegraphics[height=50mm]{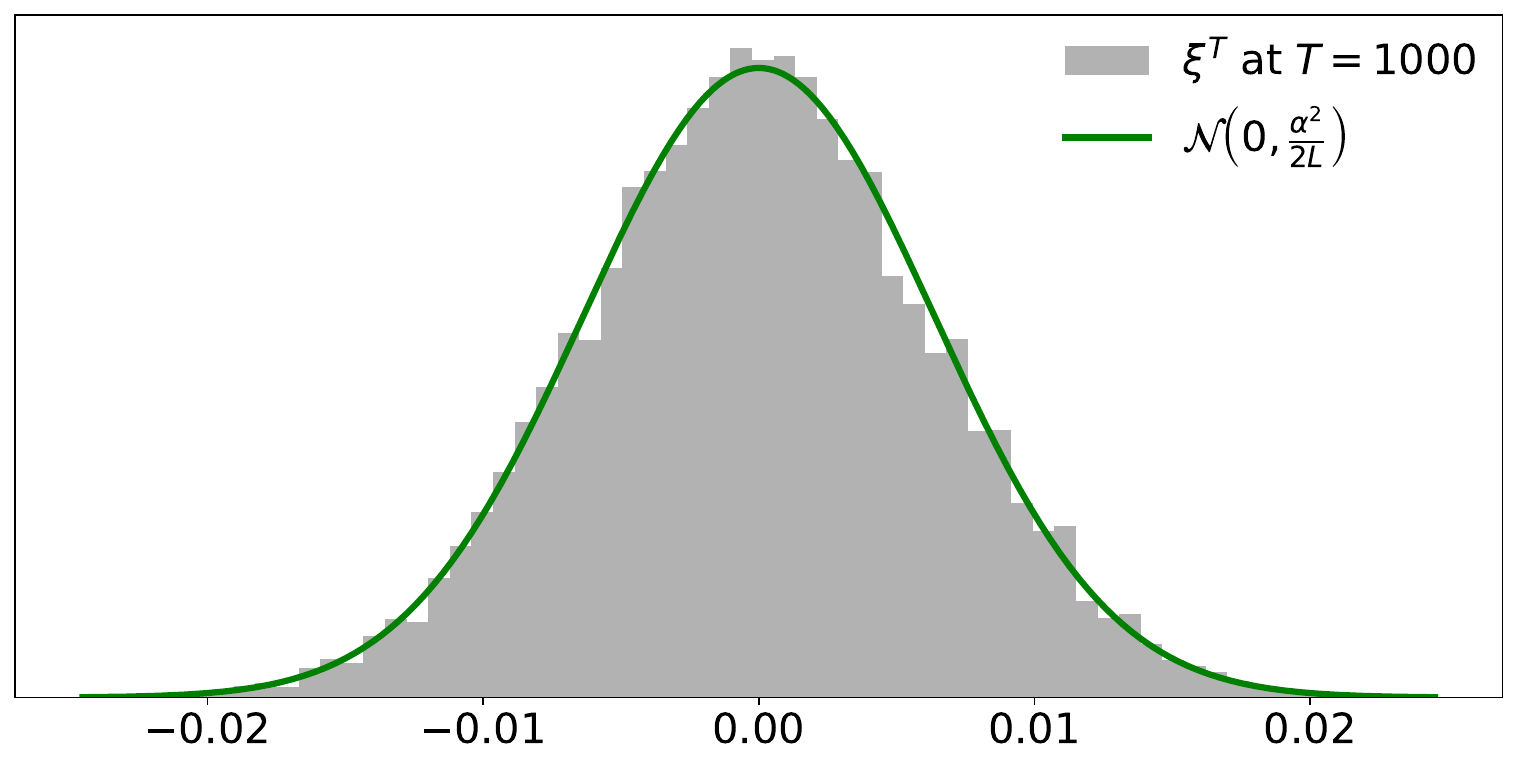} 
	\caption{Gaussian fluctuations. The grey histogram presents the fluctuations data $\{\xi_n^T\}_{n=1}^{10^4}$ corresponding to $10^4$ independent trials of network initial conditions; each $\xi_n^T$ is obtained by starting at initial condition trial $n$ and iterating for $T=1000$ times the network dynamics at coupling strength $\alpha=0.9$ on the star of size $L=10^4$. The green curve shows the probability density function of the normal distribution $\Nil(0,\frac{\alpha^2}{2L})$ with zero mean and variance $\frac{\alpha^2}{2L}$. The tight fit indicates that the fluctuation $\xi^T$ at time $T=1000$ has Gaussian statistics.}
	\label{fig:star_CLT}
\end{figure}

\subsubsection{Statement of main result on the star network} \label{sec:intro_informal_star}
Theorem A below underpins the observations made for the star network dynamics (\ref{eq:star_dynamics}): the mean-field dimensional reduction for hub behavior in Section \ref{mfs} is addressed in item (i), scaling relations for the large fluctuation frequency in Section \ref{freqs} are addressed in item (ii), and finally, the Gaussian nature of fluctuations in Section \ref{sec:gaussian_star} are addressed in item (iii).

In the statement below, we use $\mathrm{Prob}$ to denote the product probability measure on $\Omega\times\T^N$ of the Bernoulli measure for $\boldsymbol{\omega}\in \Omega=\{0,1,2,3\}^{\N\times N}$ times the volume for $x\in\T^N$.

\textbf{Theorem A: hub dynamics in star network random dynamical system.} \textit{
	Consider the dynamics (\ref{eq:star_dynamics}) on a star network with $f_{\omega},h_{\omega}$ as in Example \ref{eg:intro_star_dynamics}, and initial conditions following the uniform distribution on $\T^N$, $N=L+1\gg1$. Then, at coupling strength $\alpha\ll L^{1/2}$, the hub evolution admits mean-field dimensional reduction defined in (\ref{eq:f_alpha_omega}), namely:
	\begin{enumerate}
		\item[$\mathrm{(i)}$] \textbf{Almost sure reduction in asymptotic time:} for any $\varepsilon\gg \alpha L^{-1/2}$,
		\[
		\mathrm{Prob} \left\{(\boldsymbol{\omega},x): \liminf_{T\to+\infty}\frac{1}{T}\sum_{t=0}^{T-1} \One_{|\xi^t(\boldsymbol{\omega},x)| \leq \varepsilon}  \geq 1-\exp(-L \varepsilon^2 \alpha^{-2}/9)\right\} =1;
		\] 
		\item[$\mathrm{(ii)}$] \textbf{Small fluctuation in long time windows:} in time windows
		$$I_{t_0}^T:=\{t_0,\cdots,t_0+T-1\},~~~~T\geq \exp(L^{1-2\kappa}),~~t_0\in\N,$$
		we have successively small fluctuations
		\[
		\mathrm{Prob}\left\{(\boldsymbol{\omega},x):  \max_{t\in I_{t_0}^T} |\xi^t(\boldsymbol{\omega},x)|  \leq 3L^{-\kappa} \alpha \right\} \geq 1-\exp(-L^{1-2\kappa}),~~~~\kappa\in(0,1/2);
		\]
		\item[$\mathrm{(iii)}$] \textbf{Gaussian fluctuations: }at any time $t\in\N$, the fluctuation $\xi^t$ is approximately Gaussian, i.e.,
		$$ \mathrm{Prob}\left\{(\boldsymbol{\omega},x): \xi^t(\boldsymbol{\omega},x) \leq s\right\} \in\left[ F_L\left( s -c_1\right)-c_2, F_L\left( s+c_1 \right) +c_2\right],~~~~s\in\R,$$
		where $F_L$ denotes the cdf of the normal distribution with zero mean and variance $\alpha^2/(2L)$, and the correction constants are
		\[
		c_1=36\alpha^2 L^{-1},~~~~c_2=8 L^{-1/2}.
		\]
	\end{enumerate}
}

\begin{rem}
To state the relations among $L,\alpha,\varepsilon$ more precisely in Theorem A, we mean that there are constants $C_1,C_2,C_3>0$ such that if $L\geq C_1$, $L^{1/2}/\alpha\geq C_2$ and $\varepsilon L^{1/2}/\alpha\geq C_3$, then Items (i), (ii) and (iii) hold.
\end{rem}

In Item (i),  $\One_{|\xi^t(\boldsymbol{\omega},x)| \leq \varepsilon} $ indicates whether or not the fluctuaion $\xi^t(\boldsymbol{\omega},x)$ at time $t$ is small; the time average computes the relative frequency of small fluctuations in the time window $t=0,\cdots,T-1$; by passing to the limit inferior we obtain the asymptotic frequency of small fluctuations starting from initial data $(\boldsymbol{\omega},x)$; finally, Item (i) says that with full probability, the asymptotic frequency of small fluctuations is close to one.

In Items (i) and (ii), we have provided an upper bound for the fluctuation size and a lower bound for the asymptotic frequency; the constants $A,\gamma$ in Figure \ref{fig:star_TT_size} are not sharp and generally may depend on the coupling function. 
Theorem A is a consequence of the more general Theorem \ref{thm:star} and the uniform typicality of the random orbits of the contractions (\ref{eq:f_omega}), see Theorem \ref{thm:typicality_randomcontractions}. The derivation of the explicit constants is discussed in Appendix \ref{sec:proof_theorem_A}. We briefly discuss the proof strategy in this particular case, which will become a fundamental step in Theorem  \ref{thm:loc_star} for dimensional reduction on more complex networks.

\textbf{Main ideas of proof for Theorem {A}.} For item (i) our strategy follows three steps:
\begin{enumerate}
	\item[1.] We recast the dimensional reduction into a problem about visits to the so-called bad set, i.e., a region in the state space $\T^N$ that produces large fluctuation $\xi^t$.  
	
	\item[2.] We show that the bad set has a small size, according to large deviation results. By ergodicity, the frequency of such visits by a typical isolated orbit is as small as the size of the bad set. 
	
	\item[3.] We relate the low-degree node orbit to the isolated orbit by shadowing. The major challenge is to bridge the fundamental gap between the topologically constructed shadowing orbit and typicality in the ergodic sense. Our Theorem \ref{thm:star} treats the general case assuming compatibility of shadowing and ergodicity. In Section \ref{sec:eg_contractions} this compatibility is verified for the case of iid random iteration of contractions as node dynamics.
\end{enumerate}

Steps 1 and 2 were put forward in \cite{PvST20} and adapted to our setting. 
Our contribution in step 3 concerns 
the \textit{ergodic} properties of the shadowing orbit, a \textit{topologically} constructed object; this problem is difficult and generally open, see Remark \ref{rem:star_typicality}. We resolve this problem in Theorem A by using \textit{Breiman's ergodic theorem} together with the \textit{uniform contraction} property to ensure the typicality of random orbits for almost every noise realization independent of the initial condition. This concludes item (i).

{Item (ii) builds on step 3. By choosing the fluctuation threshold $\varepsilon=O(L^{-\kappa})$ for some $\kappa\in(0,1/2)$, we obtain an estimate for the size of bad set, which, by stationarity of the isolated random system, equals the probability that the typical isolated orbit hits the bad set at any time. The estimates follow by excluding the probability of these bad hits for each time in a consecutive window}

{Item (iii) follows from Berry-Esseen estimates together with our shadowing technique in step 3. We consider the isolated node dynamics, observed through $x\mapsto \sin 2\pi x$, as iid random variables on $\Omega\times \T^{N}$. The fluctuation $\xi^t$ is thus the ensemble average, whose Gaussian nature conforms to the classic Berry-Esseen estimates. Our result follows by comparing orbit-wise the isolated dynamics to the low degree trajectory as in step 3.}

\textbf{Resilience against local perturbation.} 
Consider a minor malfunction in the star network dynamics of one low degree node, which switches to non-hyperbolic behavior. Our reduction technique still decouples the other low degree nodes into typical shadowing orbits and obtains the same control on the fluctuation, up to an $O(L^{-1})$ loss due to the malfunctioning low degree node.

Another major advantage of this resilience of our technique is the generalizability to dynamics on more realistic networks that feature a power-law degree distribution. An important feature of many real-world networks is the power-law degree distribution, that is, the frequency $P(k)$ of nodes of degree $k$ in the network is proportional to $k^{-\beta}$ for some power-law exponent $\beta>0$.  Internet, World Wide Web, and power grids are known to have power-law degree distribution \cite{Chung_2006}. The reduced equation \ref{eq:f_alpha_omega} depends on the effective coupling strength $\alpha_i$, which is determined by $\alpha$ as well as the hub degree; see Theorem B below. The node degrees in the intermediate range give rise to a continuum of dynamical possibilities between the massively connected hub behavior and the almost isolated behavior. Nodes of a certain intermediate degree are bound to lose hyperbolicity in their mean-field reduced behavior, violating the global hyperbolicity assumption.

Our technique enables us to obtain dimensional reduction principle for realistic networks without gaps in degree distribution, such as power-law networks. In fact, many other networks are covered by our result, as long as the locally star-like property is satisfied, see Sections \ref{sec:reduction_loc_star_like} and \ref{sec:eg_loca_star_network}.

\subsection{Dynamics on power-law networks} \label{sec:intro_PL}
We use the Chung-Lu model  \cite{Chung_2006} to produce large power-law networks with well-understood graph-theoretic properties. 
To construct a connected random power-law graph $G_1$, 
we first construct a Chung-Lu random graph $G_0$ from expected degree sequence
\begin{equation} \label{eq:ChungLu_PL_w_sequence}
	\begin{aligned}
		w_i:=& \frac{\beta-2}{\beta-1} w n^{\frac{1}{\beta-1}} \left[n\left(\frac{w(\beta-2)}{m(\beta-1)}\right)^{\beta-1} +i-1\right]^{-\frac{1}{\beta-1}},~~~~i=1,\cdots,n,
	\end{aligned}
\end{equation}
where $n=10^6$ denotes the number of nodes, $\beta=3$ power-law exponent , $w=10$ average expected degree, and $m=10^3$ largest expected degree. On the empty graph consisting of nodes $1,\cdots,n$, we add links between nodes $i$ and $j$ as independent Bernoulli variables with success probability $p_{ij} = w_iw_j/\sum_{k=1}^nw_k$. The resulting graph $G_0$ is not connected, but has a giant component, which we take to be $G_1$. By concentration inequalities, we will see in Lemma \ref{thm:CLlemma5.7} that the actual degree sequence $k_i$ is concentrated at the expected version $w_i$ above; moreover, we will prove that this random power-law graph is locally star-like in the sense that most neighbors of any hub are of low degree, see Theorem B. In our simulation, $G_1$ consists of $N=998168$ nodes, with maximum degree $\Delta_0=979$ and minimum degree 1, and is a connected power-law graph. To showcase its power-law degree distribution, we plot in Figure \ref{fig:PL_degrees} left panel the degree $k$ against the frequency $P(k)$ of nodes of degree $k$ in log-log scale. For comparison, we show a power-law $k^{-3}$ in green. 

\begin{figure}[h]
	\centering 
	\begin{tabular}{cc}
		\includegraphics[height=50mm]{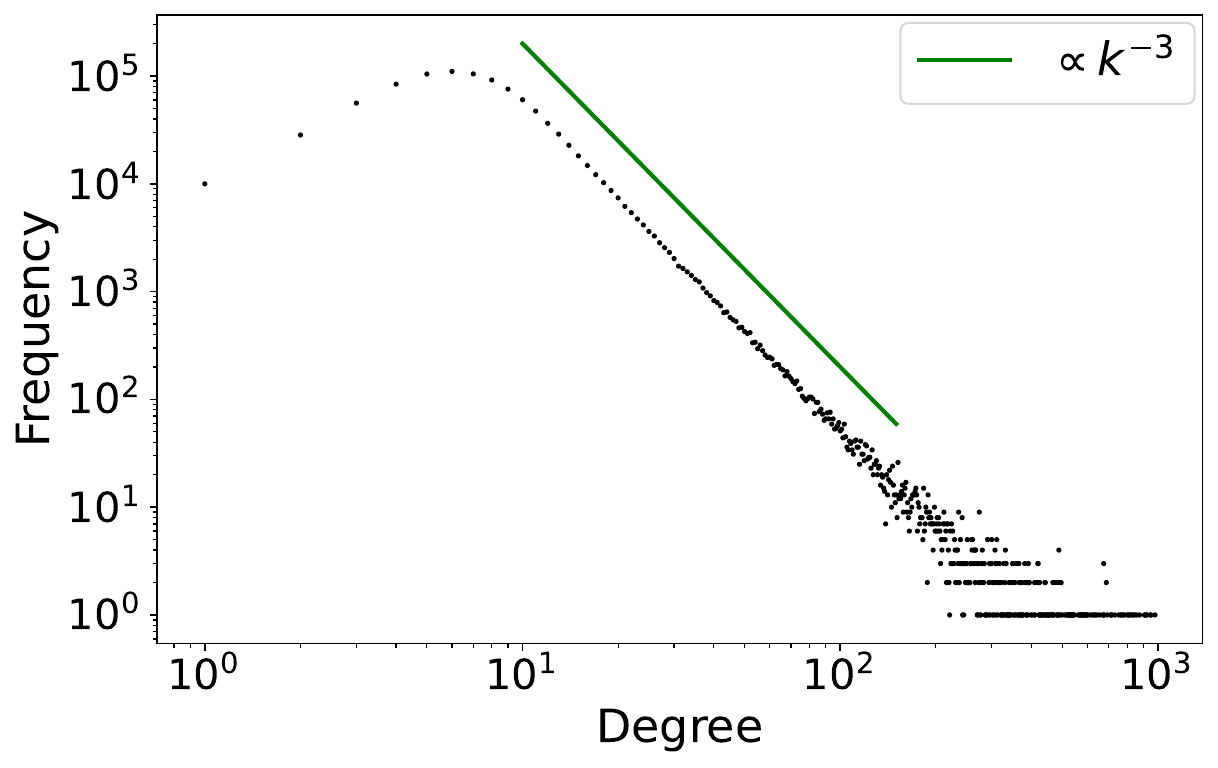}&
		\includegraphics[width=50mm]{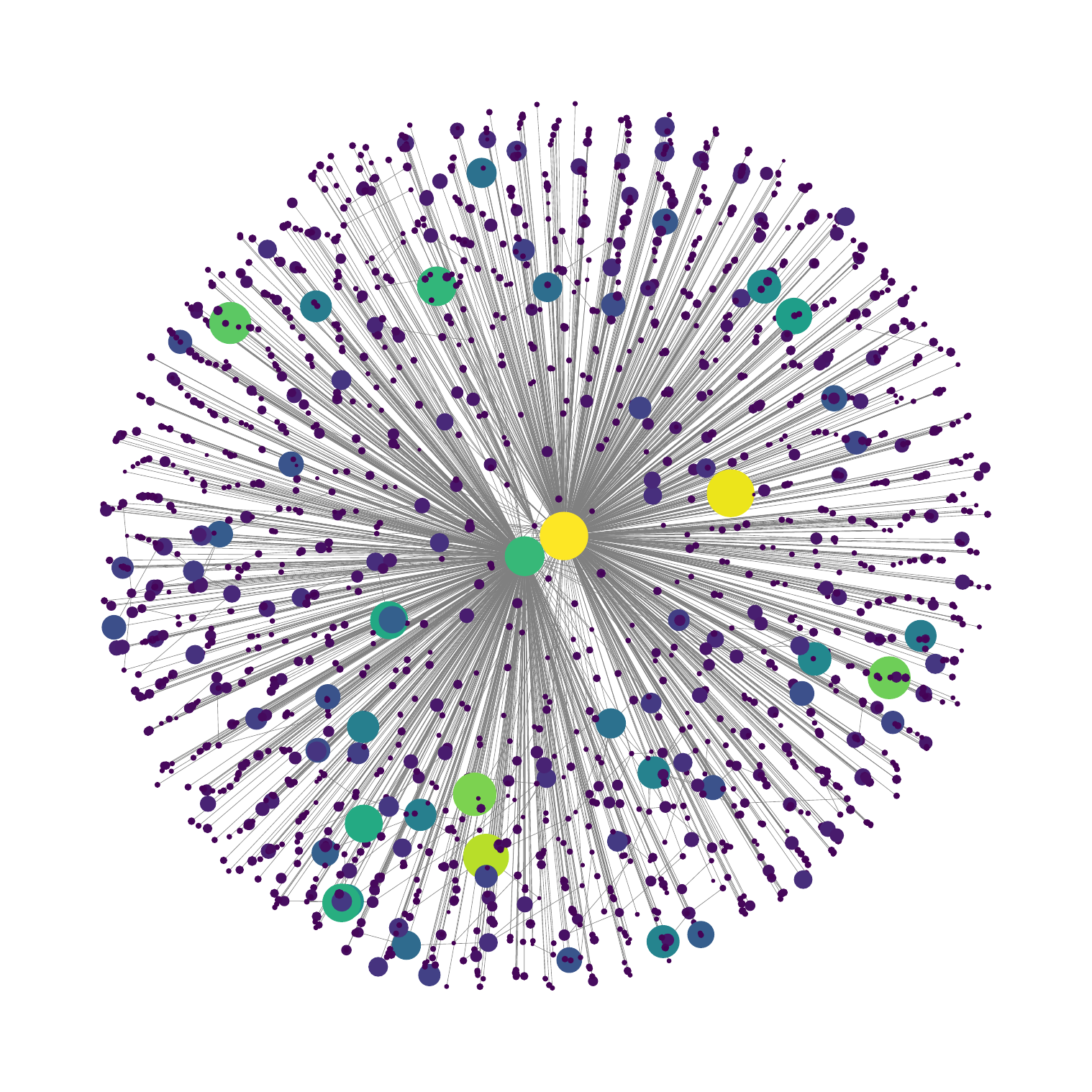}
	\end{tabular}
	\caption{Random power-law network $G_1$ generated from Chung-Lu model on $N=998168$ nodes with power-law exponent $\beta=3$, largest degree $\Delta_0=979$ and lowest degree 1. The left panel shows in log-log scale the degree distribution of $G_1$, that is, degree $k$ in horizontal axis versus the frequency $P(k)$ of nodes of degree $k$. The power-law in green highlights the fact that $P(k)\propto k^{-3}$. The right panel draws the subgraph $S$ of $G_1$ restricted to three nodes of degrees $54, 875, 979$, shown in the center, together with their neighbors in $G_1$ shown as surrounding, with node degrees reflected by size and color. This indicates that most neighbors of a hub in $G_1$ are of low degree.}
	\label{fig:PL_degrees}
\end{figure}

\textbf{Hubs $\Hil_{\Delta}$, low degree nodes $\Lil_{\delta}$, and the star-like index $\nu$.} In $G_1$ most neighbors of any hub $i$ are low degree nodes; i.e., the power-law graph is locally star-like. To illustrate, we draw a subgraph of $G_1$ by selecting three nodes of degrees $54, 875$ and $979$ respectively, shown in the center of the right panel in Figure \ref{fig:PL_degrees}, together with their neighbors shown as surrounding. The colors and sizes of the nodes reflect their degrees in $G_1$. 

Figure \ref{fig:PL_degrees} left panel shows no gap in the degree distribution; in particular, there are no natural scales to distinguish the hubs from low degree nodes, so we have to introduce them by hand. For $G_1$, we put hub scale $\Delta=900$, low degree scale $\delta=100$, and thus define the collection of \textit{$\Delta$-hubs} to be
$$\Hil_{\Delta}:=\left\{ i:  k_i>\Delta \right\}$$
and the collection of \textit{$\delta$-low degree nodes} to be
$$\Lil_{\delta}:= \left\{j:  k_j< \delta \right\};$$
we find $M:=\#\Hil_{\Delta}=7$ hubs and $L:=\#\Lil_{\delta}=995635$ low degree nodes in $G_1$. More generally, the choice of these thresholds $\Delta,\delta$ is a delicate issue and will be treated in detail later. Roughly speaking, a hub is understood as any node $i$ whose degree $k_i$ is comparable with the largest degree $\Delta_0$, whereas a low degree node $j$ has $k_j/\Delta_0 \rightarrow 0$ as  $\Delta_0$ grows.
Denote by $\Nil_i:=\{j: A_{ij}=1\}$ the set of neighbors of node $i$. We define the \textit{star-like index $\nu_i$ at hub $i\in\Hil_{\Delta}$} to be the proportion of low degree neighbors $\Nil_i\cap\Lil_{\delta}$ of hub $i$
$$\nu_i:= \frac{\#\Nil_i\cap\Lil_{\delta}}{ k_i},$$
and the \textit{star-like index $\nu$ of network $G_1$} to be the minimum star-like index among all hubs
\[
\nu:= \min\{\nu_i:i\in\Hil_{\Delta}\}.
\]
In $G_1$, we find $\nu=0.941$; in other words, more than $94.1\%$ of neighbors of each of the seven hubs in $G_1$ are of low degree. As we will prove in Theorem \ref{thm:loc_starlike_PL}, the star-like index $\nu$ of a large power-law network $G$ with exponent $\beta>2$ is close to 1, given the appropriate scales $\Delta,\delta$. 

\subsubsection{Emergent hub dynamics on power-law networks} \label{sec:emergent_PLhub_dynamics}
Using the same isolated dynamics (\ref{eq:f_omega}) and coupling function (\ref{eq:h_omega}) as in Example \ref{eg:intro_star_dynamics}, we fix coupling strength $\alpha=0.9$ and initialize the node states $(x_1^0,x_2^0,\cdots,x_N^0)\in \T^N$ randomly uniformly in $[0,1)$, then iterate the $G_1$-network dynamics (\ref{eq:network_dynamics}). We discard the first $5000$ iterates as transients and collect the next $1000$ iterates. 
In Figure \ref{fig:PL_reduction} , we select three nodes of different degrees $54,875$ and $979$ for the left, middle, and right panels, respectively, and plot in red the hub states $z^t$ against its next states $z^{t+1}$.

Note that the node behaviors vary drastically according to their degree. On the left panel, the dynamics of a node of degree 54 remain contractive; on the central panel, the node of degree 875 appears to have an expanding region; and lastly, on the right panel, the massive hub of degree 979 appears to hover around a deterministic fixed point near $z=0.2$. This shows the variety of node behaviors emergent from the interactions.

To explain, we continue to write $z_i^t$ for $x_i^t$ to emphasize the hubs $i\in\Hil_{\Delta}$. Split the coupling into contributions from low degree and non-low degree neighbors 
$$\sum_{j=1}^N A_{ij}h_{\boldsymbol{\omega}_i^t}(z_i^t,x_j^t)=\sum_{j\in \Nil_i\cap \Lil_{\delta}} h_{\boldsymbol{\omega}_i^t}(z_i^t,x_j^t)+ \sum_{j\in \Nil_i\setminus\Lil_{\delta}} h_{\boldsymbol{\omega}_i^t}(z_i^t,x_j^t);$$ the first term is the main one and sums over $\# \Nil_i\cap \Lil_{\delta}=\nu_i k_i$ contributions, whose mean can be approximated 
\[
\alpha\frac{\nu_i k_i}{\Delta_0}\frac{1}{\nu_i k_i} \sum_{j\in \Nil_i\cap\Lil_{\delta}} h_{\boldsymbol{\omega}_i^t}(z_i^t, x_j^t) \approx   \alpha\frac{\nu_i k_i}{\Delta_0} \int_{\T} h_{\boldsymbol{\omega}_i^t} (z_i^t,x)\mathrm{d}x
\]
as a space average against the Lebesgue measure. 
So we approximate
$$z_i^{t+1} = f_{\alpha_i, \boldsymbol{\omega}_i^t}(z_i^t) + \xi_i^t,$$ 
where the reduced map $f_{\alpha,\omega}$ was defined in eq. (\ref{eq:f_alpha_omega}), 
\[
\alpha_i:= \alpha \nu_i  k_i \Delta_0^{-1}
\]
is the \textit{effective coupling strength} that hub $i$ feels in the network dynamics, and the \textit{mean-field fluctuation} of hub $i$ at time $t$ from initial datum $(\boldsymbol{\omega},x)$ is, in this example,
$$\xi_i^t(\boldsymbol{\omega},x): = \frac{\alpha_i}{\nu_ik_i} \sum_{j \in \Nil_i\cap\Lil_{\delta} }\sin 2\pi x_j^t + \frac{\alpha}{\Delta_0}\sum_{j\in \Nil_i\setminus \Lil_{\delta}} \left[\sin2\pi x_j^t - \sin 2\pi z_i^t - \frac{\boldsymbol{\omega}_i^t}{3.6}\right].$$ 
The first term concerns the low-degree neighbors of the hub $i$, resembles the star case and will be controlled in a similar strategy; the second term gathers all contributions from non-low degree nodes, whose dynamics are not controlled and hence will be estimated simply as $O(1-\nu_i)$ because the sum has only $(1-\nu_i)k_i$ terms, each of which is bounded.

In Figure \ref{fig:PL_reduction}, we plot in green the graph of this one-dimensional system, and the actual node behavior in red, numerically corroborating the mean-field dimensional reduction. Theorem C (i) provides the corresponding mathematical statement.

\begin{figure}[h]
	\centering 
	\includegraphics[height=50mm]{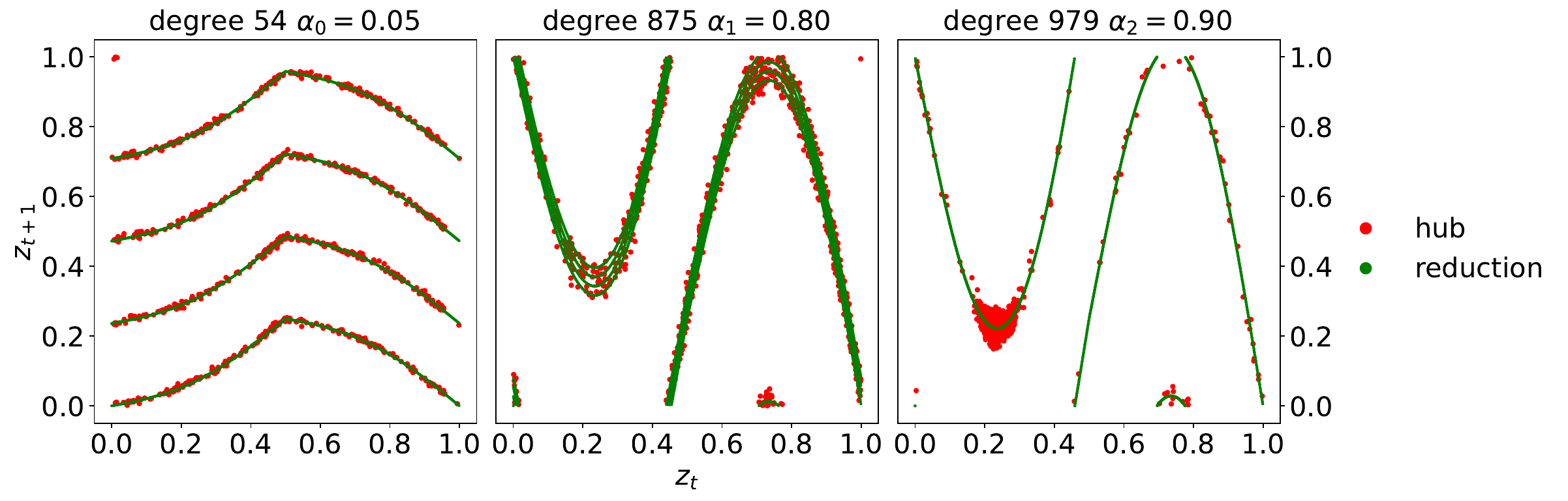}
	\caption{Hub dynamics of various effective coupling strengths. On power-law network $G_1$ with maximum degree $\Delta_0=979$, we run dynamics (\ref{eq:network_dynamics}) at fixed coupling strength $\alpha=0.9$. The left, center, right panels concern three nodes of degree $54,875,979$ respectively; each panel presents in red the node state $z^t$ versus next state $z^{t+1}$. The three nodes experience the mean-field dimensional reduction of effective coupling strengths $\alpha_i$ proportional to their degrees, plotted in green. 
	}
	\label{fig:PL_reduction}
\end{figure}


\subsubsection{System size induced desynchronization} \label{sec:size_desync}
In Figure \ref{fig:PL_excursions} we
plot in grey the time series of the \textit{desynchronization level}
\[
\eta^t=z_{i_0}^t - z_{i_1}^t,
\] 
i.e., difference of states 
between the two most massively connected hubs $i_0$ of degree $\Delta_0=979$ and $i_1$ of degree $978$.  For this we run the $G_1$-network dynamics (\ref{eq:network_dynamics}) at $\alpha=0.9$ with random initial conditions, discard the first $1500$ iterated as transients, and plot for the next $1000$ iterates. Here, large values of $\eta^t$ indicate desynchronization. 
In fact, the simulations in \cite{Ric16} revealed that this desynchronization becomes rare for large system size $N$. In a statistical mechanics system with continuum symmetry, similar size induced desynchronization effects have been characterized in  \cite{Bertini2013}.
\begin{figure}[h]
	\centering 
	\includegraphics[width=120mm]{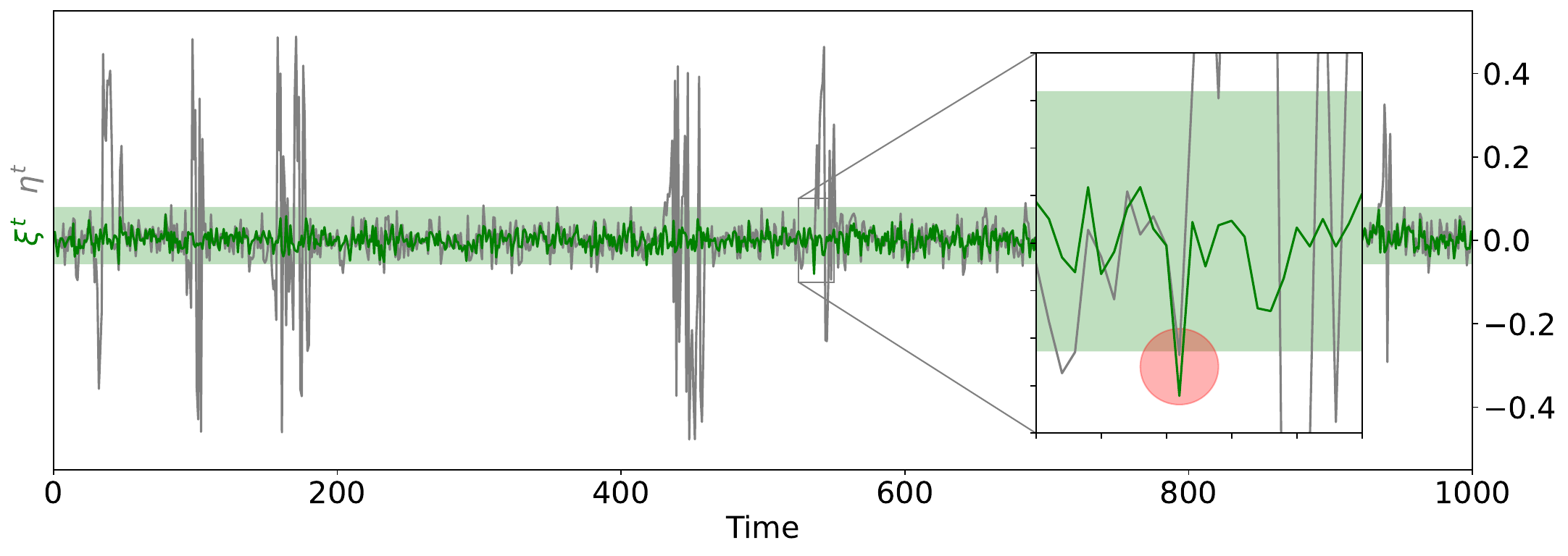} 
	\caption{System size induced desynchronization between the two most massive hubs on a power-law network. 
		We simulate the $G_1$-network dynamics on $N=99816$ nodes at coupling strength $\alpha=0.9$, and
		plot in grey the time series of hub desynchronization level $\eta^t:= z_{i_0}^t-z_{i_1}^t$ between hubs $i_0$ of degree $\Delta_0=979$ and $i_1$ of degree $978$.  Large $\eta^t$ indicate desynchronization episodes. The green time series shows comparatively small fluctuations $\xi_{i_0}^t$, with the green shaded band indicating the trapping region $[z_-,z_+]$ of the reduced dynamics $f_{0.9}$ re-centered at fixed point $z_*=f_{0.9}z_*$. The inset highlights the desynchronization mechanism, namely, an instance of a fluctuation $\xi_{i_0}^t$ sufficiently large to kick the hub $z_{i_0}^t$ out of the trapping region $[z_-,z_+]$ causes a subsequent episode of desynchronization.}
	\label{fig:PL_excursions}
\end{figure}

To explain this system size induced desynchronization, we note that at effective coupling strength $\alpha_i=0.9$, the reduced map  $f_{0.9}=f_{0.9,\omega}$, $\omega=0,1,2,3$ for the hub $i$ reads
$$f_{0.9}(z)= f_0(z) - 0.9 \sin 2\pi z.$$
It has a unique attractive fixed point $z_*\approx 0.224$, 
and nearby points in the trapping region
$[z_-,z_+]$ 
are attracted towards $z_*$ uniformly. 
Outside this region, points may enter regions of expansion by $f_{0.9}$. 

The hubs $i=i_0,i_1$ have  $\alpha_i=\alpha \nu_i\kappa_i\Delta_0^{-1}$ very close to $0.9$, thus remain in $[z_-,z_+]$, and syncronize with $|\eta^t|\leq z_+-z_-$, as long as $\xi_i^t$ is sufficiently small. 
In Figure \ref{fig:PL_excursions} the green time series for $\xi_{i_0}^t$ ocassionally kicks $z_{i_0}^t$ out of $[z_-,z_+]$, the re-centered version $[z_--z_*, z_+ - z_*]$ shown as the green shaded band, resulting in large $\eta^t$. 
The inset highlights one instance of this desynchronization mechanism. 



The precise nature of the system size induced desynchronization is related to the central limit behavior of $\xi_{i}^t$ in Section \ref{sec:gaussian_star}. In fact, the fluctuations $\xi_i^t$ of any hub $i\in\Hil_{\Delta}$ satisfy similar scaling relations and Gaussian statistics as in the star network case, with the star size $L$ replaced by $\nu_ik_i$, see Theorem C below. In fact, these numerical phenomena are also observed in examples beyond the setting of our Theorems, see \cite{baptista2012collective,Ric16}. 



\subsubsection{Statement of main result on the power-law network} \label{sec:intro_informa_PL}
Our next results formalize the numerical observations above, namely, the locally star-like property of large power-law networks and the mean-field dimensional reduction therein with statistically controlled fluctuations.

In the limit $a\to+\infty$ or $a\to 0^+$, we use the Bachmann–Landau big-O notation $f(a)=O(g(a))$ for $\limsup |f(a)|/g(a) <+\infty$, small-o notation $f(a)=o(g(a))$ for $\lim f(a)/g(a)=0$, same-order notation $f(a) \asymp g(a)$ for the conjunction of $f(a)=O(g(a))$ and $g(a)=O(f(a))$, and asymptotic equivalence notation $f(a)\sim g(a)$ for $\lim f(a)/g(a)=1$.

\textbf{Theorem B: power-law network is locally star-like.} \textit{Fix parameters $c_{\mathrm{hub}},\lambda_{\mathrm{ldn}}\in(0,1)$. Consider a large Chung-Lu network $G$ generated from the power-law expected degree sequence given in (\ref{eq:ChungLu_PL_w_sequence}) with $N$ nodes, power-law exponent $\beta>2$, maximum expected degree $m\asymp  N^{\frac{1}{\beta-1}}$, and mean expected degree $w= o(N^{\lambda_{\mathrm{ldn}}\frac{1}{\beta-1}})$. By setting hub and low degree scales
	$$\Delta=c_{\mathrm{hub}} m,~~~~\delta\asymp \Delta^{\lambda_{\mathrm{ldn}}},$$
	we regard nodes of degree above $\Delta$  as hubs and below $\delta$ as low degree. Then, with probability  $1-O(N^{-1/5})$, 
	we have
	$$M\sim w^{\beta-1},~~~~L\sim N,~~~~\nu= 1 - O\left(N^{-\lambda_{\mathrm{ldn}}\frac{\beta-2}{\beta-1}} w^{\beta-2}\right),$$
	where $M$ is the number of hubs, $L$ the number of low degree nodes, and $\nu$ the star-like index of $G$.}

\textbf{Main ideas of proof for Theorem B.} From concentration inequalities, namely, Chernoff bounds, \cite[Lemma 5.7]{Chung_2006} can be adapted to show that the entire actual degree sequence is concentrated around the expected version. By counting, we show the locally star-like property for the expected degree sequence, and by concentration, we obtain the same for the actual degree sequence. The subtlety lies in the careful choice of the hub and low degree scales $\Delta,\delta$ respectively. The precise statement and full proof can be found in Theorem \ref{thm:loc_starlike_PL}.

\textbf{Theorem C: hub dynamics in power-law network random dynamical system.}
\textit{ Consider dynamics (\ref{eq:network_dynamics}) on a large power-law network $G$ as in Theorem B and $f_{\omega},h_{\omega}$ as in Example \ref{eg:intro_star_dynamics}. Then, with initial conditions following the uniform distribution on $\T^N$, at coupling strength $\alpha = o\left(  \min\{ N^{\lambda_{\mathrm{ldn}}\frac{\beta-2}{\beta-1}} , N^{(1-\lambda_{\mathrm{ldn}})/(2\beta-2)}\} \right)$, each hub $i\in\Hil_{\Delta}$ admits the mean-field dimensional reduction $f_{\alpha_i,\omega}$ at effective coupling strength $\alpha_i$ defined to be
	\begin{align*}
		f_{\alpha_i, \omega}(z):= f_{\omega}(z) + \alpha_i\int_{\T} h_{\omega}(z, x)\mathrm{d}x\mod 1,~~~~\alpha_i:= \alpha \nu_i  k_i\Delta_0^{-1}.
	\end{align*}
}
More precisely, we have:
\begin{enumerate}
	\item[$\mathrm{(i)}$] \textbf{Almost sure reduction in asymptotic time:} for any
	$$\varepsilon\geq \max\left\{ \frac{17}{2} \alpha (1-\nu), 34\pi\alpha^2\delta / \Delta \right\},~~~~1-\nu=O\left(N^{-\lambda_{\mathrm{ldn}}\frac{\beta-2}{\beta-1}} w^{\beta-2}\right),~~\Delta \asymp N^{\frac{1}{\beta-1}},~~\delta \asymp \Delta^{\lambda_{\mathrm{ldn}}} \asymp N^{\frac{\lambda_{\mathrm{ldn}}}{\beta-1}},$$
	we have
	\[
	\mathrm{Prob} \left\{ (\boldsymbol{\omega},x): \liminf_{T\to+\infty} \frac{1}{T}\sum_{t=0}^{T-1} \One_{\max\{|\xi_i^t(\boldsymbol{\omega},x)|:i\in\Hil_{\Delta}\}\leq \varepsilon} \geq 1-  \exp\left(-\nu\Delta\varepsilon^2\alpha^{-2}/19\right) \right\} =1;
	\]
	\item[$\mathrm{(ii)}$] \textbf{Small fluctuation in long time windows:} in time windows 
	$$I_{t_0}^T:=\{t_0,\cdots,t_0+T-1\},~~~~T\geq \exp(\Delta^{1-2\kappa}),~~t_0\in\N,$$
	we have successively small fluctuations
	\[
	\mathrm{Prob} \left\{ (\boldsymbol{\omega},x): \max_{i\in\Hil_{\Delta}} \max_{t\in I_{t_0}^T} |\xi_i^t(\boldsymbol{\omega},x)| \leq 3\alpha \Delta^{-\kappa}  \right\} \geq 1-\exp(-\Delta^{1-2\kappa}), ~~~~0<\kappa <\min\left\{1/2, 1-\lambda_{\mathrm{ldn}}, \lambda_{\mathrm{ldn}}\frac{\beta-2}{\beta-1}\right\};
	\]
	\item[$\mathrm{(iii)}$] \textbf{Gaussian fluctuations:} at any time $t\in\N$, the fluctuation $\xi_i^t$ is approximately Gaussian, i.e.,
	$$\mathrm{Prob}\{(\boldsymbol{\omega},x):\xi_i^t(\boldsymbol{\omega},x)\leq s\}  \in [F_i(s-s_1) - s_2, F_i(s+s_1)+s_2],~~~~i\in\Hil_{\Delta},~~s\in\R,$$
	where  $F_i$ denotes the cdf of the normal distribution with zero mean and variance $\alpha_i^2/(2\nu_i k_i)$, and the correction constants are $$s_1=O(N^{-\frac{1-\lambda_{\mathrm{ldn}}}{\beta-1}} + N^{-\lambda_{\mathrm{ldn}} \frac{\beta-1}{\beta-2}} w^{\beta-2}),~~~~s_2=O(N^{-\frac{1}{2\beta-2}}).$$
\end{enumerate}

\textbf{Main ideas of proof for Theorem C.} We combine Theorems A and B. The locally star-like property from Theorem B ensures that the hub and low degree scales are separated, allowing the low degree orbit to be shadowed by an isolated orbit; again by Theorem B, most hub neighbors are of low degree and hence well-controlled by the shadowing technique from Theorem A. When excluding the bad sets for all hubs, we use Theorem B to ensure that there are only $M\sim w^{\beta-1}$ hubs. The detailed proof of Theorem C can be found in Appendix \ref{sec:proof_theorem_C}.

Our abstract result Theorem \ref{thm:loc_star} treats the general case on a locally star-like network. In Section \ref{sec:eg_loca_star_network} the locally star-like property is verified for the random power-law network model.


We organize the rest of the paper as follows. Section \ref{sec:reduction_star} spells out all hypotheses, and develops the abstract dimensional reduction principle on the star network, which showcases the essential arguments of our shadowing plus typicality technique. Section \ref{sec:reduction_loc_star_like} introduces the notion of locally star-like networks and extends the dimensional reduction technique from star to locally star-like networks. Section \ref{sec:eg_loca_star_network} constructs the random power-law network model and proves that it is locally star-like, Theorem B. Finally in Section \ref{sec:eg_contractions} we return to the case of iid random iteration of contractions as node dynamics and complete the proof of Theorems A and C.

\textbf{Acknowledgments.} 
We thank Matthias Wolfrum and Edmilson Roque dos Santos for valuable discussions and comments. We gratefully acknowledge support of the UK Royal Society through a Newton Advanced Fellowship  (NAF$\backslash$R1$\backslash$180236). TP and ZB were supported in part by FAPESP Grant Nos. Cemeai 2013/07375-0, 2018/26107-0, 2021/11091-3, and by the Serrapilheira Institute (Grant No. Serra-1709-16124). TP thanks the Humboldt Foundation for support through a Bessel Fellowship. JSWL has been supported by the EPSRC (grants EP/Y020669/1 and EP/S023925/1) and the JST (Moonshot R \& D Grant  JPMJMS2021 of  IRCN (Tokyo)). He also thanks GUST (Kuwait) for research support.

\section{Dimensional reduction} 


Fix a probability preserving transformation $(\Omega,\Fil,\PP, \theta)$ as the common model for environmental noise and a network $G$ with maximum degree $\Delta_0$. On each node $i=1,,\cdots,N$, the local dynamics are given by a random dynamical system $\varphi_i$ on the circle $\T=\R/\Z=[0,1]/0\sim 1$
$$\varphi_i:\N\times \Omega\times\T\to\T,~~~~(t,\omega,x)\mapsto \varphi_i(t,\omega,x),$$
where $\varphi_i$ 
\begin{enumerate}
	\item[(i)] is measurable with respect to $2^{\N}\otimes \Fil\otimes \Bil(\T)$ and $\Bil(\T)$;
	\item[(ii)] satisfies the cocycle property over $\theta$, namely,  $\varphi_i(0,\omega,\cdot)=\mathrm{id}_{\T}$ for each $\omega\in\Omega$ and $\varphi_i(t+s,\omega,\cdot)=\varphi_i(t,\theta^s\omega,\cdot)\circ\varphi_i(s,\omega,\cdot)$ for all $s,t\in\N$ and $\omega\in\Omega$.
\end{enumerate}

Each node $i$ influences its neighbor $j\in\Nil_i$ with contribution $h(\omega,x_j,x_i)$ given by a random pairwise coupling function $h:\Omega\times\T^2\to\R$. At coupling strength $\alpha>0$, the network dynamical system is the RDS
\begin{equation}\label{eq:Gnetwork_dynamics}
	\Phi_{\alpha}:\N\times\Omega\times \T^N\to\T^N,
\end{equation}
where each node $(\Phi_{\alpha}(t,\omega,x))_i=x_i^t$ evolves by
\begin{equation}\label{eq:Ghub_dynamics}
	x_i^{t+1} = \varphi_i(1,\theta^t\omega,x_i^t) + \alpha\cdot \frac{1}{\Delta_0} \sum_{j=1}^N A_{ij}h(\theta^t\omega,x_i^t,x_j^t)\mod 1,~~~~i=1,\cdots,N,
\end{equation}
where $A=(A_{ij})_{i,j}$ is the same adjacency matrix of network $G$ as in Eq. (\ref{eq:network_dynamics}).
Note that the trajectory $\{x_j^t:t\in\N\}$ of low degree node $j\in\Lil_{\delta}$ is a pseudo-orbit of $\varphi_j$ with error at each time step bounded by $\alpha\Delta_0^{-1} \delta \|h(\theta^t\omega,\cdot,\cdot)\|_{C^0}$.

We assume the following conditions on the random dynamical system (\ref{eq:Gnetwork_dynamics}):
\begin{enumerate}
	\item[(R1)] Independent and identically distributed node maps: the random variables $\omega\mapsto \varphi_i(1,\theta^t\omega,\cdot)$, $i=1,\cdots,N$, $t\in\N$ take values in the space of continuous circle maps, have the same distribution, and are independent in time $t$ and node $i$; 
	
	\item[(R2)] Unique stationary measure of node dynamics: the isolated node dynamics $\varphi_i$, $i=1,\cdots,N$ admits a unique stationary measure $m$;
	
	\item[(R3)] $C^4$ pairwise coupling maps: $h(\theta^t\omega,\cdot,\cdot)$, $t\in\N$ share the same distribution in the space $ C^4(\T^2;\R)$ of $C^4$-smooth maps $\T^2\to\R$ and are independent in time $t$. 
\end{enumerate}

\begin{rem}
	In fact, we may allow in (R1--2) that the low-degree node maps $\varphi_j$ share the same map distribution with unique stationary measure $m$, while the other node maps enjoy different distributions. The smoothness $C^4$ in (R3) is assumed to ensure sufficient decay of the Fourier coefficients.  In (R3) we may also allow the coupling maps $h_i(\omega,x_i,x_j)$ to vary among nodes $i$, as long as all neighbors $j\in\Nil_i$ of any node $i$ influence it via the same coupling map $h_i$.  Under assumptions (R1--3) the network dynamical system $\Phi_{\alpha}$ is an iid random iteration of continuous maps $\Phi_{\alpha}(1,\theta^t\omega,\cdot)$ on $\T^{N}$.
\end{rem}

\begin{rem}[Notation of $\omega$]
	To avoid notational cluter, we have changed the notation of noise. Now we only use the simple font $\omega$ as an element of the abstract probability space $\Omega$ to denote noise realization. The fact that the noise is iid in time $t=0,1,\cdots$ and node coordinates $i=1,\cdots,N$ is reflected in node dynamics $\varphi_i(t,\omega,\cdot)$ and coupling function $h(\theta^t\omega,\cdot,\cdot)$.
\end{rem}

\subsection{Reduction on the star network}\label{sec:reduction_star}
When $G$ is a star graph consisting of 1 hub and $L=N-1$ low degree nodes, the hub 
evolves by
\begin{equation}\label{eq:hub}
	z^{t+1}= \varphi_1(1,\theta^t\omega,z^t)+ \alpha\cdot\frac{1}{L}\sum_{j=2}^N h(\theta^t \omega,z^t,x_j^t)\mod 1,
\end{equation}
and each  low-degree node 
evolves by
\begin{equation}\label{eq:ldn}
	x_j^{t+1} = \varphi_j(1,\theta^t\omega,x_j^t)  + \alpha \cdot \frac{1}{L} h(\theta^t\omega,x_j^t,z^t) \mod1,~~~~j=2,\cdots,N.
\end{equation}

The mean-field reduction seeks to approximate the mass action of the $L$  low-degree nodes on the hub $z$ by the space average against $m$ on $\T$
$$\alpha \frac{1}{L} \sum_{j=2}^N h(\omega,z,x_j)= \alpha\int_{\T} h(\omega,z,y)\mathrm{d} m(y) +  \xi_m(\omega,x),$$
where the corresponding fluctuation $\xi_m(\omega,x)$ is given by
\begin{equation}\label{eq:fluctuationXi}
	\alpha^{-1}\xi_m(\omega,x)= \frac{1}{L} \sum_{j=2}^N h(\omega,z,x_j) - \int_{\T} h(\omega,z,y)\mathrm{d}m(y).
\end{equation}

This way, the hub behavior in the $N$-dimensional network dynamics on the undirected star becomes reduced to (approximated by) the one-dimensional random system $\varphi_{\alpha,m}:\N\times\Omega\times \T\to \T$, recursively defined by
\begin{equation}\label{eq:reducedhub}
	\varphi_{\alpha,m}(t+1,\omega,z)=\varphi_1(1,\theta^t\omega,\varphi_{\alpha,m}(t,\omega,z)) +\alpha\int_{\T} h(\theta^t\omega,\varphi_{\alpha,m}(t,\omega,z),y)\mathrm{d} m(y).
\end{equation}

In this notation, the hub evolution becomes
$$z^{t+1}= \varphi_{\alpha,m}(1,\theta^t\omega,z^t) + \xi_m(\theta^t\omega,x^t).$$

\begin{defn}\label{def:reduction}
	We say that the \textit{hub (\ref{eq:hub}) admits $\varepsilon$-reducion to $\varphi_{\alpha,m}$ on initial data $(\omega,x)\in\Omega\times \T^N$ at time $t$} if 
	$$| \xi_m(\theta^t\omega,\Phi_{\alpha}(t,\omega,x))|\leq \varepsilon,$$
and that it admits $\varepsilon$-reduction to $\varphi_{\alpha,m}$ on initial data $(\omega,x)\in\Omega\times \T^N$ with exceptional frequency at most $\rho$ if 
	\[
	\limsup_{T\to+\infty} \frac{1}{T} \# \left\{t\in[0,T-1]: | \xi_m(\theta^t\omega,\Phi_{\alpha}(t,\omega,x))|>\varepsilon\right\} \leq  \rho.
	\]
\end{defn}

\begin{theorem}[Reduction theorem on a star]\label{thm:star}
	Suppose (R1--3) hold for the star network random dynamical system (\ref{eq:hub}-\ref{eq:ldn})  on $N-1$  low-degree nodes and one hub node at coupling strength $\alpha>0$. Let
	initial data $(\omega,x)\in\Omega\times \T^{N}$ be such that the network trajectory $\{x^t=\Phi_{\alpha}(t,\omega,x): t\in\N \}$ admits an $m$-typical shadowing orbit in the low degree coordinates; more precisely, there is $(\omega_{{s}},x_{{s}})\in\Omega\times\T^{N}$ satisfying
	\begin{equation}\tag{Shadowing}
	\sup_{t\in\N}\max_{j=2,\cdots,N}	d_{\T}(x_j^t, \varphi_j(t,\omega_{{s}},x_{{s},j}))\leq \varepsilon_s \left(\alpha (N-1)^{-1} \sup_{t\in\N } \|h(\theta^t\omega,\cdot,\cdot)\|_{C^0}\right); 
	\end{equation}
	\begin{equation}\tag{Typicality}
		\frac{1}{T}\sum_{t=0}^{T-1} \delta_{\varphi_2(t,\omega_{{s}},x_{{s},2})} \otimes\cdots\otimes \delta_{\varphi_N(t,\omega_{{s}},x_{{s},N})}\xrightarrow[T\to+\infty]{\text{weak}^*} m^{\otimes (N-1)},
	\end{equation}
	where $d_{\T}$ denotes the distance on the circle, the shadowing precision $\delta_s\mapsto\varepsilon_s (\delta_s )$ is a $\R_+$-valued function converging to 0 as $\delta_s $ tends to 0. Then, for any error tolerance 
	$$\varepsilon\geq  \max\left\{ \alpha (N-1)^{-1/2},4\alpha \sup_{t\in\N }|h(\theta^t \omega,\cdot,\cdot)|_{\mathrm{Lip}} \varepsilon_s \left( \alpha (N-1)^{-1} \sup_{t\in\N } \|h(\theta^t\omega,\cdot,\cdot)\|_{C^0} \right)\right\},$$
	the hub behavior (\ref{eq:hub}) admits $\varepsilon$-reducion to $\varphi_{\alpha,m}$ on initial data $(\omega,x)\in\Omega\times \T^N$ with exceptional asymptotic frequency at most $\rho$ with
	$$ \rho(\varepsilon,\omega)=D(\varepsilon,\omega) \exp(-(N-1)\varepsilon^2 \alpha^{-2} c(\omega)),~~~~D(\varepsilon,\omega)\asymp \varepsilon^{-1} \sup_{t\in\N }\|h(\theta^t\omega,\cdot,\cdot)\|_{C^4},~~c(\omega)=\frac{450}{\sup_{t\in\N } \|h(\theta^t\omega,\cdot,\cdot)\|_{C^4}},$$
	where the constants $D(\varepsilon,\omega)$ and $c(\omega)$ are independent of $N$ and $x$.
\end{theorem}

\begin{rem} \label{rem:star_typicality}
	(Typicality) condition assumes that the shadowing orbit $(\varphi_2(t,\omega_{{s}},x_{{s},2}),\cdots,\varphi_N(t,\omega_{{s}},x_{{s},N}))$ is $m^{\otimes L}$-typical, $L=N-1$, i.e., follows the distribution $m^{\otimes L}$ given by the unique stationary meausre. Note that Theorem \ref{thm:star} (Shadowing) plus (Typicality) for one initial data $(\omega,x)\in \Omega\times \T^{N}$ and concludes the reduction for this particular instance of initial data. 
	To obtain Shadowing and Typicality conditions for $\PP\otimes m^{\otimes N}$-a.e. $(\omega,x)\in\Omega\times \T^N$ as in Theorem A, one tends to encounter the following scenario. By local hyperbolicity of node dynamics $\varphi_i$,  
	the shadowing technique produces some initial datum $(\omega_s,x_s)$, whose random orbit $\varphi_i(t,\omega_s,x_s)$ achieves (Shadowing). And we wish to establish (Typicality) for $\varphi_i(t,\omega_s,x_s)$. For this, the existence and uniqueness of stationary measure $m$ from (R2) ensures, see \cite[Theorem 16.4]{Bhattacharya2022}, the ergodicity of the Markov chain associated to the iid random iteration $\varphi_i$ on each node $i=1,\cdots,N$ with initial distribution $m$.
	\begin{enumerate}
		\item By Breiman's ergodic theorem \cite{Breiman1960}, for each $x_i\in\T$, there is $\Omega_{x_i}\subseteq\Omega$ with $\PP(\Omega_{x_i})=1$ such that the random orbit $\varphi_i(t,\omega,x_i)$ is $m$-typical for every $\omega\in \Omega_{x_i}$. However, this typicality falls short because there may be pathological situations where the shadowing initial datum $(\omega_s,x_{s,i})$ is such that $\omega_s \notin \Omega_{x_{s,i}}$.
This is the case with random expanding or hyperbolic maps as node dynamics.

		\item For random \textit{uniform contractions}, such as those in Example \ref{eg:intro_star_dynamics},
		 Theorem \ref{thm:typicality_randomcontractions} establishes a uniform $\Omega_*$ with $\PP(\Omega_*)=1$ such that the random orbit $\varphi_i(t,\omega,x_i)$ is $m$-typical for every $(\omega,x_i)\in \Omega_*\times\T$. Moreover, we show that $(\omega_s,x_{s})=(\omega,x)$, see Corollary \ref{cor:star_contractions}. Thus, Shadowing and Typicality conditions are achieved for every $(\omega,x)\in \Omega_*\times\T^N$.
	\end{enumerate} 
\end{rem}

\begin{rem}
	If $\PP\otimes m^{\otimes N}$-a.e. $(\omega,x)\in\Omega\times \T^N$ admits shadowing intial data $(\omega_{{s}},x_{{s}})$ satisfying (Shadowing) and (Typicality), then the $\varepsilon$-reduction holds $\PP\otimes m^{\otimes N}$-almost surely by Theorem \ref{thm:star}. 
	In the case of Example \ref{eg:intro_star_dynamics}, we have $(\omega,x)= (\omega_{{s}},x_{{s}})$ and hence obtain
	(ii) Small fluctuation in long time windows and (iii) Gaussian fluctuations, as in Theorem A. See Appendix \ref{appendix:ac_shad} for more details.
\end{rem}

\begin{rem}
As in the usual Shadowing Lemmas, $\epsilon_s(\delta_s)$ in Theorem \ref{thm:star} is the shadowing precision which upper bounds the distance between the shadowing and pseudo orbits; it is a function of the error tolerance $\delta_s$ of the pseudo orbit. In Example \ref{eg:intro_star_dynamics}, we have $\delta_s=\alpha L^{-1}$ and $\epsilon_s(\delta_s) = \frac{\delta_s}{1-\lambda}= \frac{\alpha L^{-1}}{1-\lambda}$, where $\lambda=1/2$ is the contraction rate, see also Corollary \ref{cor:star_contractions}. Our result relies on the hyperbolicity of the low degree node dynamics alone in terms of the shadowing property, which allows us to overcome the shortcoming i) of \cite{PvST20} as discussed in Introduction.
\end{rem}

\subsubsection{Decay of Fourier coefficients}
We start with some preparations in Fourier analysis. Write the Fourier series of the coupling maps $h(\omega,\cdot,\cdot)\in C^4(\T^2,\R)$
$$h(\omega,x_1,x_2)= \sum_{(n_1,n_2)\in \Z^2}a^{\omega}_{(n_1,n_2)} e^{2\pi i(n_1x_1+n_2x_2)},$$
where the Fourier coefficients are defined by
\begin{align*}
	a^{\omega}_{(n_1,n_2)}:= &\int_0^1 \int_0^1 e^{-2\pi i(n_1x_1+n_2x_2)} h(\omega,x_1,x_2)\mathrm{d}x_1\mathrm{d}x_2.
\end{align*}

We gather some basic facts of decay of Fourier coefficients for smooth functions from \cite[Theorem 3.3.9]{Grafakos2014}.
\begin{lemma}[Decay of Fourier coefficients]\label{cor:decayFouriercoeff}
	For $h(\omega,\cdot,\cdot)\in C^4(\T^2;\R)$, $(n_1,n_2)\in\Z^2$ and multi-index $(m_1,m_2)\in\N^2$ with $\sum_{j=1}^2 m_j\leq 4$, we have
	$$ (2\pi)^{m_1+m_2}|n_1|^{m_1} |n_2|^{m_2} |a^{\omega}_{(n_1,n_2)}|\leq \|h(\omega,\cdot,\cdot)\|_{C^4}.$$
	In particular, we obtain
	\begin{align*}
		|a^{\omega}_{(0,0)}|\leq &\|h(\omega,\cdot,\cdot)\|_{C^4};\\
		|a^{\omega}_{(0,n_2)}|\leq & \frac{1}{(2\pi)^4 n_2^4} \|h(\omega,\cdot,\cdot)\|_{C^4},~~~~\forall n_2\neq 0;\\
		|a^{\omega}_{(n_1,0)}|\leq & \frac{1}{(2\pi)^4 n_1^4} \|h(\omega,\cdot,\cdot)\|_{C^4},~~~~\forall n_1\neq 0;\\
		|a^{\omega}_{(n_1,n_2)}|\leq& \frac{1}{(2\pi)^4 n_1^2 n_2^2}\|h(\omega,\cdot,\cdot)\|_{C^4},~~~~\forall n_1\neq0\text{ and }n_2\neq 0,
	\end{align*}
	so that its Fourier series 
	\begin{align*}
		h(\omega,x_1,x_2) = \sum_{(n_1,n_2)\in\Z^2} a^{\omega}_{(n_1,n_2)} e^{2\pi i (n_1x_1+n_2x_2)} =& a^{\omega}_{(0,0)} + \sum_{n_2\neq 0} a^{\omega}_{(0,n_2)} e^{2\pi i n_2x_2} + \sum_{n_1\neq 0} a^{\omega}_{(n_1,0)} e^{2\pi i n_1x_1}  \\
		&+ \sum_{n_1\neq 0 \text{ and } n_2\neq 0} a^{\omega}_{(n_1,n_2)} e^{2\pi i (n_1x_1+n_2x_2)} 
	\end{align*}
	converges absolutely and uniformly.
\end{lemma}

\subsubsection{Bad sets}
For $\phi\in C^0(\T;\R)$, we define its bad set to be
$$B(\varepsilon ,\phi):=\left\{{x}\in \T^N: \left|\frac{1}{L} \sum_{j=2}^N \phi({x}_j)- \int_{\T} \phi\mathrm{d}m \right|>\varepsilon\right\}.$$
In the Introduction example, the bad set $B(\varepsilon,\phi)$ with $\phi(x)=\sin2\pi x$ is precisely the part of state space that produces large fluctuation $|\xi^t|>\varepsilon$. As we will see shortly, the bad sets play a similar role in the general case and thus good control on the frequency of visits to them is key to obtaining $\varepsilon$-reduction. 
We first estimate the size of the bad set and frequency of visit to it by trajectories of the product system.


\begin{prop}[Size of Bad Set]\label{prop:sizebadset}
	The bad set $B(\varepsilon,\phi)$ has exponentially small $m^{\otimes N}$-volume:
	$$m^{\otimes N} (B(\varepsilon,\phi))\leq 2\exp(-L2^{-1}\|\phi\|_{\infty}^{-2} \varepsilon^2).$$
\end{prop}

\begin{proof}
	For $i=2,\cdots, N$, define random variable $X_i$ on $(\T^N, m^{\otimes N})$ by 
	$$ X_i:\T^N\to\T,~~~~(x_1,\cdots,x_N)\mapsto \phi(x_i).$$
	Then, these random variables $X_2\cdots,X_N$ are independent and bounded.  By Hoeffding's Inequality \cite[Theorem 2.2.2]{Vershynin2018}, for any $\varepsilon>0$, we have
	$$m^{\otimes N} (B(\varepsilon,\phi))\leq 
	2\exp(-L2^{-1}\|\phi\|_{\infty}^{-2} \varepsilon^2),$$
	as desired.
\end{proof}

\begin{rem}
	Note that the geometric structure of the bad set $B(\varepsilon,\phi)$ may be complicated, depending on the choice of $\phi$. However, its $m^{\otimes L}$-volume has an exponentially small upper bound, regardless of $\phi$. 
\end{rem}

\begin{lemma}[Random Orbit Visits Bad Set with Small Frequency]\label{lemma:RDSvisitfreq} Under the hypotheses of Reduction Theorem \ref{thm:star}, the uncoupled orbit of the shadowing pair $(\omega_{{s}},x_{{s}})\in \Omega\times \T^{N}$ visits the bad set with exponentially small frequency:
	$$\limsup_{T\to+\infty} \frac{1}{T}\sum_{t=0}^{T-1} \One_{B(\varepsilon,\phi)}  (\varphi_2(t,\omega_{{s}},x_{{s},2}),\cdots,\varphi_N(t,\omega_{{s}},x_{{s},N}) )\leq 2\exp(-L2^{-1}\|\phi\|_{\infty}^{-2} \varepsilon^2).$$
\end{lemma}

\begin{proof}
	We emphasize only the dependence on $\varepsilon$ and suppress other dependencies by writing $B_{\varepsilon}$ for $B(\varepsilon,\phi)$. First we treat the case where the bad set has null boundary:
	$$m^{\otimes L}(\partial B_{\varepsilon})=0.$$
	From the weak$^*$ convergence in (Typicality) assumption in Theorem \ref{thm:star}, we conclude the frequency of visit by Portmanteau Theorem and Proposition \ref{prop:sizebadset}
	$$\lim_{T\to+\infty} \frac{1}{T}\sum_{t=0}^{T-1} \One_{B_{\varepsilon}}  (\varphi_2(t,\omega_{s},x_{{s},2}),\cdots,\varphi_N(t,\omega_{s},x_{{s},N}) ) = m^{\otimes L}(B_{\varepsilon})\leq 2\exp(-L2^{-1}\|\phi\|_{\infty}^{-2} \varepsilon^2).$$
	
	Note that the bad set is the superlevel set of the continuous function $\T^L\to\R$, $x\mapsto \left|\frac{1}{L}\sum_{j=2}^N \phi(x_j) - \int_{\T}\phi\mathrm{d}m\right|$, and hence $B_{\varepsilon}= \bigcup_{\varepsilon'\in[0,\varepsilon)} \partial B_{\varepsilon'}$ is a disjoint union. Since $m^{\otimes L}(B_{\varepsilon})<\infty$, it follows that $m^{\otimes L}(\partial B_{\varepsilon'})>0$ for at most countably many $\varepsilon'\in[0,\varepsilon)$. Hence, there are $\varepsilon'\in(0,\varepsilon)$ arbitrarily close to $\varepsilon$ with $m^{\otimes L}(\partial B_{\varepsilon'})=0$.
	
	Fix any such $\varepsilon'\in(0, \varepsilon)$. We then have $B_{\varepsilon}\subseteq B_{\varepsilon'}$ and $\One_{B_{\varepsilon}}\leq \One_{B_{\varepsilon'}}$. Therefore, from the special case of null boundary, we obtain
	\begin{align*}
		&\limsup_{T\to+\infty} \frac{1}{T}\sum_{t=0}^{T-1} \One_{B_{\varepsilon}}  (\varphi_2(t,\omega_{{s}},x_{{s},2}),\cdots,\varphi_N(t,\omega_{{s}},x_{{s},N}) )  \\
		\leq& \limsup_{T\to+\infty} \frac{1}{T}\sum_{t=0}^{T-1} \One_{B_{\varepsilon'}}(\varphi_2(t,\omega_{{s}},x_{{s},2}),\cdots,\varphi_N(t,\omega_{{s}},x_{{s},N}) )= m^{\otimes L}(B_{\varepsilon'})\\
		\leq& 2\exp(-L2^{-1}\|\phi\|_{\infty}^{-2} (\varepsilon')^2).
	\end{align*}
	But since $\varepsilon'\in(0,\varepsilon)$ can be chosen arbitrarily close to $\varepsilon$, the estimates follow.
\end{proof}

When dealing with multiple bad sets $B_1,\cdots,B_D$, each with asymptotic frequency at most $\rho_k$ of the trajectory visiting $B_k$, $k=1,\cdots,D$, we can lower bound the asymptotic frequency of the trajectory visiting none of the bad sets by $ 1- \sum_{k=1}^D \rho_k$.

\subsubsection{Proof of reduction theorem on a star}
We will split the fluctuation estimates
\begin{align*}
	\xi_{m}(\theta^t\omega,x^t) =&  \alpha\frac{1}{L} \sum_{j=2}^N h(\theta^t\omega,z^t,x_j^t) - \alpha\int_{\T} h(\theta^t\omega,z^t,y)\mathrm{d}m(y)
\end{align*}
into three parts. By (Shadowing), the first term 
\begin{align*} 
	\zeta_d(t,\omega,x,\omega_{{s}}, x_{{s}}) :=& \alpha\frac{1}{L} \sum_{j=2}^N \left[h(\theta^t\omega, z^t,x_j^t) - h(\theta^t\omega,z^t,x_{{s},j}^t) \right], \tag*{(decoupling)}
\end{align*}
decouples the low degree node orbit $x_j^t=\Phi_{\alpha}(t,\omega,x)_j$ by comparing it with the isolated shadowing orbit $x_{{s},j}^t = \varphi_j(t,\omega_{{s}}, x_{{s},j})$,
which will be controlled by the shadowing precision and Lipschitz continuity of $h(\theta^t\omega,\cdot,\cdot)$, see Lemma \ref{lemma:zeta1}.
We now have
$$\xi_{m}(\theta^t\omega,x^t)= \zeta_d(t,\omega,x,\omega_{{s}},x_{{s}})+ \alpha\left[  \frac{1}{L} \sum_{j=2}^N  h(\theta^t\omega,z^t,x_{{s},j}^t) - \int_{\T} h(\theta^t\omega,z^t,y)\mathrm{d}m(y)\right].$$
By (Typicality), the shadowing orbit $\left(x_{{s},2}^t,\cdots,x_{{s},N}^t\right)$ distributes in time as $m^{\otimes L}$ and thus visits any bad set with controlled asymptotic frequency. In order to find the appropriate bad sets, we take Fourier series expansion 
\[
h(\theta^t\omega,x_1,x_2) = \sum_{(n_1,n_2)\in\Z^2} a^{\theta^t\omega}_{(n_1,n_2)} \phi_{n_1}(x_1)\phi_{n_2}(x_2),
\]
where $\phi_n(x):=e^{2\pi i nx}$ and note that the Fourier modes corresponding to $n_2=0$ drop out 
$$a_{(n_1,0)}^{\theta^t\omega} \phi_{n_1}(z^t)\left[\phi_0(x_{{s},j}) - \int_{\T} \phi_0(y)\mathrm{d}m(y)\right]=0.$$ 
The high Fourier modes $\phi_{n_2}$ with $|n_2|>D$ for a suitable cutoff level $D$, make up the second term
\begin{align*}
	\zeta_h (t,\omega,x,\omega_{{s}}, x_{{s}}) :=&  \alpha \sum_{|n_2|> D}  \left[a_{(0,n_2)}^{\theta^t\omega}   +  \sum_{n_1\neq0} a_{(n_1,n_2)}^{\theta^t\omega} \phi_{n_1}(z^t)  \right] \left[\frac{1}{L} \sum_{j=2}^N\phi_{n_2}(x_{{s},j}^t) - \int_{\T} \phi_{n_2}(y) \mathrm{d}m(y)\right],\tag*{(HFM)}
\end{align*}
which will be controlled by decay of Fourier coefficients, see Lemma \ref{lemma:zeta2}..
The low Fourier modes $\phi_{n_2}$ with $1\leq |n_2|\leq D$ make up the third term
\begin{align*}
	\zeta_{\ell} (t,\omega,x,\omega_{{s}}, x_{{s}}) := & \alpha \sum_{1\leq |n_2|\leq D}  \left[a_{(0,n_2)}^{\theta^t\omega}   +  \sum_{n_1\neq0} a_{(n_1,n_2)}^{\theta^t\omega} \phi_{n_1}(z^t)  \right]\ \left[\frac{1}{L} \sum_{j=2}^N\phi_{n_2}(x_{{s},j}^t) - \int_{\T} \phi_{n_2}(y) \mathrm{d}m(y)\right],\tag*{(LFM)}
\end{align*}
which will be controlled by rare visits of the shadowing orbit $\left(x_{{s},2}^t,\cdots,x_{{s},N}^t\right)$  to the bad sets $B(\varepsilon_b,\phi_{n_2}), 1\leq |n_2|\leq D$, see Lemma \ref{lemma:zeta3}. 
In the next three lemmas, we estimate each of $\zeta_d,\zeta_h,\zeta_{\ell}$.

\begin{lemma}[Shadowing estimates]\label{lemma:zeta1}
	Under the hypotheses of Theorem \ref{thm:star}, we control the $\mathrm{(decoupling)}$ term 
	\[
	\left|\zeta_d(t,\omega,x,\omega_{{s}}, x_{{s}}) \right|\leq \varepsilon/4,~~~~t\in\N.
	\]
\end{lemma}

\textit{Proof of Lemma \ref{lemma:zeta1}.} By Lipschitz continuity of $h(\omega,\cdot,\cdot)\in C^4(\T^2,\R)$, we have
\begin{align*}
	\left|\zeta_d (t,\omega,x,\omega_{{s}}, x_{{s}})  \right|\leq& \alpha \frac{1}{L} L |h(\theta^t\omega,z^t,\cdot)|_{\mathrm{Lip}} d_{\T}(x_j^t, x_{{s},j}^t) \leq \alpha  \sup_{t\in\N } |h(\theta^t\omega,\cdot,\cdot)|_{\mathrm{Lip}} \varepsilon_s (\alpha L^{-1} \sup_{t\in\N } \|h(\theta^t\omega,\cdot,\cdot)\|_{C^0}).
\end{align*}
The estimate follows by the assumption
$\varepsilon\geq  4\alpha \sup_{t\in\N }|h(\theta^t\omega,\cdot,\cdot)|_{\mathrm{Lip}} \varepsilon_s(\alpha L^{-1} \sup_{t\in\N } \|h(\theta^t\omega,\cdot,\cdot)\|_{C^0}).$ \qed

\begin{lemma}[High Fourier modes controlled by decay of Fourier coefficients]\label{lemma:zeta2}
	Under the hypotheses of Theorem \ref{thm:star}, by choosing $D(\varepsilon,\omega) \asymp \varepsilon^{-1}\sup_{t\in\N }\|h(\theta^t\omega,\cdot,\cdot)\|_{C^4}$, we estimate $\mathrm{(HFM)}$ 
	\[
	\left|\zeta_h (t,\omega,x,\omega_{{s}}, x_{{s}})\right| \leq \varepsilon/2,~~~~\forall t\in\N.
	\]
\end{lemma}

\textit{Proof of Lemma \ref{lemma:zeta2}.} By decay of Fourier coefficients Lemma \ref{cor:decayFouriercoeff}, we estimate (HFM)
\begin{align*}
	\left|\zeta_h(t,\omega,x,\omega_{{s}}, x_{{s}})\right|  \leq& \alpha \sum_{|n_2|>D} \frac{\|h(\theta^t\omega,\cdot,\cdot)\|_{C^4}}{ (2\pi)^4 n_2^4}  \frac{1}{L} L2\|\phi_{n_2}\|_{\infty} + \alpha \sum_{n_1\neq0}\sum_{|n_2|>D} \frac{\|h(\theta^t\omega,\cdot,\cdot)\|_{C^4}}{(2\pi)^4 n_1^2 n_2^2} \frac{1}{L} L 2 \|\phi_{n_2}\|_{\infty} \\
	\leq& \frac{2\alpha \sup_{t\in\N }\|h(\theta^t\omega,\cdot,\cdot)\|_{C^4}}{ (2\pi)^4}\sum_{|n_2|>D} 1/n_2^4 + \frac{2\alpha \sup_{t\in\N } \|h(\theta^t\omega,\cdot,\cdot)\|_{C^4}}{(2\pi)^4} \sum_{n_1\neq 0}\frac{1}{n_1^2} \sum_{|n_2|>D}\frac{1}{n_2^2} \\
	=& \frac{2\alpha \sup_{t\in\N }\|h(\theta^t\omega,\cdot,\cdot)\|_{C^4}}{ (2\pi)^4}\sum_{|n_2|>D} 1/n_2^4 + \frac{\alpha \sup_{t\in\N } \|h(\theta^t\omega,\cdot,\cdot)\|_{C^4}}{24\pi^2}  \sum_{|n_2|>D}\frac{1}{n_2^2},
\end{align*}
where we have used the identity $\sum_{n=1}^{\infty}1/n^2=\pi^2/6$ in the last equality. By choosing $D\geq \max\{D_1(\varepsilon,\omega),D_2(\varepsilon,\omega)\}$, where $D_1(\varepsilon,\omega)$ is so large that 
$$\frac{2\alpha \sup_{t\in\N }\|h(\theta^t\omega,\cdot,\cdot)\|_{C^4}} { (2\pi)^4} \sum_{|n|>D_1(\varepsilon,\omega)}1/n^4\leq \varepsilon/4,$$
and 
$D_2(\varepsilon,\omega)$ is so large that $$\frac{\alpha \sup_{t\in\N }\|h(\theta^t\omega,\cdot,\cdot)\|_{C^4}} { 24\pi^2}\left(\sum_{|n|>D_2(\varepsilon,\omega)}1/n^2\right)\leq \varepsilon/4,$$ 
we obtain that $\left|\zeta_h(t,\omega,x,\omega_{{s}}, x_{{s}})\right| \leq \varepsilon/2.$ 
Since $\sum_{|n|>D} 1/n^4 \asymp D^{-3}$, we take $D_1(\varepsilon,\omega)\asymp \varepsilon^{-1/3}\sup_{t\in\N }\|h(\theta^t\omega,\cdot,\cdot)\|_{C^4}$; since $\sum_{|n|>D} 1/n^2 \asymp D^{-1}$, we take $D_2(\varepsilon,\omega)\asymp \varepsilon^{-1}\sup_{t\in\N }\|h(\theta^t\omega,\cdot,\cdot)\|_{C^4}$. \qed

\begin{lemma}[Low Fourier modes controlled by rare visits to bad sets]\label{lemma:zeta3}
	Under the hypotheses of Theorem \ref{thm:star}, by choosing $D(\varepsilon,\omega) \asymp \varepsilon^{-1} \sup_{t\in\N }\|h(\theta^t\omega,\cdot,\cdot)\|_{C^4}$, we control $\mathrm{(LFM)}$ 
	\[
	\left|\zeta_{\ell} (t,\omega,x,\omega_{{s}}, x_{{s}}) \right|\leq \varepsilon/4
	\]
	with exponentially small exceptional asymptotic frequency at most $\rho$ with 
	$$\rho(\varepsilon,\omega)= 4D (\varepsilon,\omega) \exp\left(-L \frac{450\varepsilon^2}{\alpha^2\sup_{t\in\N }\|h(\theta^t\omega,\cdot,\cdot)\|^2_{C^4}}\right).$$
\end{lemma}

\textit{Proof of Lemma \ref{lemma:zeta3}.}
We estimate $\zeta_{\ell} (t,\omega,x,\omega_{{s}}, x_{{s}})$ as a problem of frequency of visits to the bad sets $B_n:= B(\varepsilon_b, \phi_n)$, $n=\pm1,\cdots,\pm D$. For each such $n$, Lemma \ref{lemma:RDSvisitfreq} implies that the shadowing orbit $x_{{s}}^t = (\varphi_2(t,\omega_{{s}}, x_{{s},2}),\cdots,\varphi_N(t,\omega_{{s}}, x_{{s},N}))$ visits $B_n$ with exponentially small frequency at most
$ 2\exp(-L 2^{-1} \|\phi_n\|_{\infty}^{-2}\varepsilon_b^2)=2\exp(-L 2^{-1} \varepsilon_b^2)$. 
Whenever all such bad sets are avoided, we can bound (LFM)
\begin{align*}
	\left|\zeta_{\ell} (t,\omega,x,\omega_{{s}}, x_{{s}}) \right| \leq& \alpha \sum_{1\leq |n_2|\leq D} \left[|a_{(0,n_2)}^{\theta^t\omega}|   +  \sum_{n_1\neq0} |a_{(n_1,n_2)}^{\theta^t\omega}|  \right]\left|\frac{1}{L} \sum_{j=2}^N\phi_{n_2}(x_{{s},j}^t) - \int_{\T} \phi_{n_2}(y) \mathrm{d}m(y)\right|\\
	\leq& \alpha \sum_{1\leq |n_2|\leq D} \left[ \frac{\sup_{t\in\N }\|h(\theta^t\omega,\cdot,\cdot)\|_{C^4}}{(2\pi)^4 n_2^4} + \sum_{n_1\neq 0} \frac{\sup_{t\in\N }\|h(\theta^t\omega,\cdot,\cdot)\|_{C^4}}{(2\pi)^4 n_1^2n_2^2} \right] \varepsilon_b\\
	\leq& \frac{\alpha\sup_{t\in\N }\|h(\theta^t\omega,\cdot,\cdot)\|_{C^4}}{(2\pi)^4 } \left[\frac{2\pi^4}{90} + \frac{2\pi^2}{6}\frac{2\pi^2}{6} \right]\varepsilon_b = \frac{\alpha\sup_{t\in\N }\|h(\theta^t\omega,\cdot,\cdot)\|_{C^4}}{120 } \varepsilon_b.
\end{align*}
By choosing 
\begin{equation}\label{eq:epsilon_E}
	\varepsilon_b=\frac{30}{\alpha\sup_{t\in\N }\|h(\theta^t\omega,\cdot,\cdot)\|_{C^4}}\varepsilon,
\end{equation}
we obtain $\left|\zeta_{\ell} (t,\omega,x,\omega_{{s}}, x_{{s}}) \right|\leq \varepsilon/4$ with exponentially small exceptional asymptotic frequency at most $\rho$ with 
$$\rho(\varepsilon,\omega)= 2D(\varepsilon,\omega) 2\exp(-L 2^{-1} \varepsilon_b^2) =  4D(\varepsilon,\omega)  \exp\left(-L \frac{450\varepsilon^2}{\alpha^2\sup_{t\in\N }\|h(\theta^t\omega,\cdot,\cdot)\|^2_{C^4}}\right),$$
where $D(\varepsilon,\omega) = \max\{D_1(\varepsilon,\omega),D_2(\varepsilon,\omega)\}\asymp \varepsilon^{-1} \sup_{t\in\N }\|h(\theta^t\omega,\cdot,\cdot)\|_{C^4}$. \qed

\textit{Proof of Theorem \ref{thm:star}.} Fix, as in the hypotheses of Theorem \ref{thm:star},
$$\varepsilon\geq  \max\left\{ \alpha L^{-1/2},4\alpha \sup_{t\in\N }|h(\theta^t \omega,\cdot,\cdot)|_{\mathrm{Lip}} \varepsilon_s \left( \alpha L^{-1} \sup_{t\in\N } \|h(\theta^t\omega,\cdot,\cdot)\|_{C^0} \right)\right\},$$
so that the (decoupling) term $\left|\zeta_d(t,\omega,x,{{s}},x_{{s}})\right|\leq \varepsilon/4$ by Lemma \ref{lemma:zeta1}. Now fix $D(\varepsilon,\omega) = \max\{D_1(\varepsilon,\omega),D_2(\varepsilon,\omega)\}$, so that (HFM) term $\left|\zeta_h(t,\omega,x,\omega_{{s}},x_{{s}})\right| \leq\varepsilon/2$ by Lemma \ref{lemma:zeta2} and (LFM) $\left|\zeta_{\ell}(t,\omega,x,\omega_{{s}},x_{{s}})\right| \leq\varepsilon/4$ with exponentially small exceptional asymptotic frequency  at most $\rho$ with 
\begin{align*}
	\rho(\varepsilon,\omega)= 4D(\varepsilon,\omega) \exp\left(-L \frac{450\varepsilon^2}{\alpha^2\sup_{t\in\N }\|h(\theta^t\omega,\cdot,\cdot)\|^2_{C^4}}\right),
\end{align*} 
by Lemma \ref{lemma:zeta3}. It is also with this exponentially small exceptional asymptotic frequency $\rho$ that we conclude
\begin{align*}
	\left|\xi_{m}(\theta^t\omega,x^t)\right| \leq& \left|\zeta_d(t,\omega,x,\omega_{{s}},x_{{s}})\right| +\left|\zeta_h(t,\omega,x,\omega_{{s}},x_{{s}})\right| + \left|\zeta_{\ell}(t,\omega,x,\omega_{{s}},x_{{s}})\right|\\
	\leq& \varepsilon/4 + \varepsilon/2 + \varepsilon/4 = \varepsilon\tag*{\qed}
\end{align*}

\begin{rem}
	In our estimates we control the fluctuation by shadowing in $|\zeta_d|\leq \varepsilon/4$, decay of Fourier coefficients in $|\zeta_h|\leq \varepsilon/2$, and rare visits to the bad set in $|\zeta_{\ell}|\leq \varepsilon/4$. The choice to split $\varepsilon$ into these three portions has an impact on the outcome of constants such as $D(\varepsilon,\omega)$, $c(\omega)$ and $\rho$. The choice of $\varepsilon_b$ in (\ref{eq:epsilon_E}) impacts the exceptional asymptotic frequency $\rho$. Such choices are made only for convenience to illustrate the scaling relation among system size $L$, fluctuation size $\varepsilon$, and exceptional asymptotic frequency $\rho$. 
\end{rem}

\subsection{Reduction on a locally star-like network}\label{sec:reduction_loc_star_like}
As illustrated for the power-law network in the introduction, a hub looks locally like a star, in the sense that most of its neighbors are low degree nodes. In this section, we provide a definition to quantify this local feature.
\begin{defn}[locally star-like network]
	Let $G=(V,E)$ be an undirected graph on $N$ nodes $V=\{1,\cdots,N\}$ indexed so that the degree sequence is non-increasing $k_1\geq \cdots\geq k_N$. 
	Choose $\Delta\in\N$ as the hub scale and let $\Hil_{\Delta}:=\{i\in V: k_i\geq \Delta\}=\{1,\cdots,M\}$ denote the collection of $M$ hubs. Choose $\delta\in\N$ as the low degree scale and let $\Lil_{\delta}:=\{i\in V: k_i\leq\delta\}=\{N-L+1,\cdots,N\}$ denote the collection of $L$ low degree nodes. We say that $G$ is a \textit{$(\Delta,\delta,\nu)$-locally star-like network} if most hub neighbors are low degree nodes in the sense that
	$$\nu_i:= \frac{\# \Nil_i\cap \Lil_{\delta}}{\#\Nil_i} \geq \nu,~~~~\forall i\in \Hil_{\Delta},$$
	where $\Nil_i:= \{j\in V: A_{ij}=1\}$ denotes the neighbors of node $i$.
\end{defn}

\begin{rem}
	The star on $L$ low degree nodes and one hub is locally star-like with $(\Delta,\delta,\nu)=(L,1,1)$. In a heterogeneous network, we will typically have $\nu$ near 1 and that the low degree scale $\delta$ is dominated by the hub degree scale $\Delta$ in the sense that $\delta/\Delta\to0$ as $N\to+\infty$. There may be intermediate nodes that are neither hubs nor low degree nodes.
\end{rem}

Now let $G$ be a $(\Delta,\delta,\nu)$-locally star-like network on $N$ nodes and consider $G$-network dynamics (\ref{eq:Gnetwork_dynamics}).
The mean-field reduction seeks to approximate the mass action of the low-degree neighbors $\Nil_i\cap\Lil_{\delta}$ of a hub $z_i$, $i\in \Hil_{\Delta}$ as a space average against some probability measure $m$ on $\T$
$$\alpha \frac{\nu_ik_i}{\Delta_0} \frac{1}{\nu_ik_i}\sum_{j=1}^N A_{ij}h(\omega,z_i,x_j) = \alpha\frac{\nu_i k_i}{\Delta_0} \int_{\T} h(\omega,z_i,y)\mathrm{d}m(y)+\xi_{i,m}(\omega,x),$$
where the fluctuation $\xi_{i,m}(\omega,x)$ is given by
$$\alpha^{-1}\xi_{i,m}(\omega,x):= \frac{1}{\Delta_0} \sum_{j=1}^N A_{ij}h(\omega,z_i,x_j) -  \frac{\nu_i k_i}{\Delta_0} \int_{\T} h(\omega,z_i,y)\mathrm{d}m(y).$$




\begin{theorem}[Reduction theorem on a locally star-like network]\label{thm:loc_star}
	Suppose (R1--3) hold for the $G$-network random dynamical system (\ref{eq:Gnetwork_dynamics}) on a $(\Delta,\delta,\nu)$-locally star-like network $G$ on $N$ nodes. Let
	initial data $(\omega,x)\in\Omega\times \T^{N}$ be such that the network trajectory $\{x^t=\Phi_{\alpha}(t,\omega,x): t\in\N \}$ admits an $m$-typical shadowing orbit in the low degree coordinates; more precisely, there is $(\omega_{{s}},x_{{s}})\in\Omega\times\T^{N}$ satisfying
	\begin{equation}\tag{Shadowing}
	\sup_{t\in\N}\max_{j\in\Lil_{\delta}}	d_{\T}(x_j^t, \varphi_j(t,\omega_{{s}},x_{{s},j}))\leq \varepsilon_s \left(\alpha\Delta_0^{-1} \delta\sup_{t\in\N} \|h(\theta^t\omega,\cdot,\cdot)\|_{C^0}\right); 
	\end{equation}
	\begin{equation}\tag{Typicality}
		\frac{1}{T}\sum_{t=0}^{T-1} \bigotimes_{j\in\Lil_{\delta}} \delta_{\varphi_j(t,\omega_{{s}},x_{{s},j})}\xrightarrow[T\to+\infty]{\text{weak}^*} m^{\otimes L},
	\end{equation}
	where $d_{\T}$ denotes the distance on the circle, the shadowing precision $\delta_s\mapsto\varepsilon_s (\delta_s )$ is a $\R_+$-valued function converging to 0 as $\delta_s $ tends to 0. Then, for any error tolerance 
	\[
	\varepsilon\geq \max\left\{\alpha \Delta^{-1/2}, 4\alpha \sup_{t\in\N }|h(\theta^t\omega,\cdot,\cdot)|_{\mathrm{Lip}} \varepsilon_s \left(\alpha\Delta_0^{-1} \delta\sup_{t\in\N} \|h(\theta^t\omega,\cdot,\cdot)\|_{C^0}\right)\right\} ,
	\]
	each hub $i\in\Hil_{\Delta}$ admits $(\varepsilon + \alpha(1-\nu)\sup_{t\in\N} \|h(\theta^t\omega,\cdot,\cdot)\|_{C^0})$-reduction to $\varphi_{\alpha_i,m}$ in Eq. (\ref{eq:reducedhub}) with $\alpha_i=\alpha\frac{\nu_i\kappa_i}{\Delta_0}$ on initial data $(\omega,x)\in\Omega\times \T^N$ with exceptional asymptotic frequency at most $\rho$ with  
	$$\rho(\varepsilon,\omega)= 4MD(\varepsilon,\omega) \exp(-\nu\Delta\varepsilon^2 \alpha^{-2} c(\omega)),~D(\varepsilon,\omega)\asymp \varepsilon^{-1}\sup_{t\in\N}\|h(\theta^t\omega,\cdot,\cdot)\|_{C^4},~c(\omega) = \frac{450}{\sup_{t\in\N }\|h(\theta^t\omega,\cdot,\cdot)\|^2_{C^4}},$$
	where $M$ denotes the number of hubs in $G$ and the constants $D(\varepsilon,\omega)$ and $c(\omega)$ are the same as in Theorem \ref{thm:star} independent of $N$ and $x$.
\end{theorem}


We will prove this theorem in a similar way as we did the reduction on a star Theorem \ref{thm:star}. The main modification concerns the bad sets for multiple hubs and the control of non-low-degree neighbors of the hubs.

\subsubsection{Proof of reduction theorem on a locally star-like network}
\textit{Proof of Theorem \ref{thm:loc_star}.}
We wiill split the fluctuation $\xi_{i,m}(\theta^t\omega,x^t)$ at hub $i\in\Hil_{\Delta}$ 
\begin{align*}
	\xi_{i,m}(\theta^t\omega,x^t) =&  \frac{\alpha}{\Delta_0} \sum_{j\in \Nil_i} h(\theta^t\omega,z_i^t,x_j^t) - \alpha\frac{\nu_i k_i}{\Delta_0}\int_{\T} h(\theta^t\omega,z_i^t,y)\mathrm{d}m(y) 
\end{align*}
into four components. By (Shadowing), the first term 
\begin{align*} \tag*{(decoupling)}
	\zeta_{i,d}(t,\omega,x,\omega_{{s}},x_{{s}}) :=& \frac{\alpha}{\Delta_0} \sum_{j\in\Nil_i\cap\Lil_{\delta}} \left[ h(\theta^t\omega,z_i^t,x_j^t) - h(\theta^t\omega,z_i^t,x_{{s},j}^t) \right]
\end{align*}
decouples the low-degree node orbits $x_j^t=\Phi_{\alpha}(t,\omega,x)_j, j\in\Nil_i\cap \Lil_{\delta}$ by comparing with the isolated shadowing orbits $x_{{s},j}^t =\varphi_j(t,\omega_{{s}},x_{{s},j})$, which will be controlled by the shadowing precision and Lipschitz continuity of $h(\theta^t\omega,\cdot,\cdot)$, see Lemma \ref{lemma:zeta_i1}. The second term gathers  contribution from the non-low-degree neighbors
\begin{align*}\tag*{(non~ldn)}
	\zeta_{i,c}(t,\omega,x,\omega_{{s}},x_{{s}}) := & \frac{\alpha}{\Delta_0} \sum_{j\in \Nil_i\setminus\Lil_{\delta}} h(\theta^t\omega,z_i^t,x_j^t),
\end{align*}
which will be controlled by the star-like index $\frac{\# \Nil_i\setminus\Lil_{\delta}}{\Delta_0}\leq 1-\nu\approx0$, see Lemma \ref{lemma:zeta_i2}. This is a new term that was not present in the proof of Theorem \ref{thm:star}. Now we have
$$\xi_{i,m}(\theta^t\omega,x^t) \leq \zeta_{i,d} + \zeta_{i,c} + \alpha \frac{\nu_1k_1}{\Delta_0}\left[ \frac{1}{\nu_ik_i}\sum_{j\in\Nil_i\cap\Lil_{\delta}}h(\theta^t\omega,z_i^t,x_{{s},j}^t) -\int_{\T}h(\theta^t\omega,z_i^t,y)\mathrm{d}m(y)\right].$$
By (Typicality) the shadowing orbit $\left(x_{{s},j}^t\right)_{j\in\Lil_{\delta}}$ distributes in time as $m^{\otimes L}$ and thus visits any bad set with controlled asymptotic frequency. To find the appropriate bad sets, we take Fourier series expansion for $h(\omega,x_1,x_2)$ and note that the Fourier modes corresponding to $n_2=0$ drop out, similar to the proof of Theorem \ref{thm:star}.

Again we gather the remaining high Fourier modes $\phi_{n_2}$ with $|n_2|>D$ above some cutoff level $D$
\begin{align*}\tag*{(HFM)}
	\zeta_{i,h}(t,\omega,x,\omega_{{s}},x_{{s}}) :=& \alpha \frac{\nu_i k_i}{\Delta_0}\sum_{|n_2|>D}  \left[a_{(0,n_2)}^{\theta^t\omega}   +  \sum_{n_1\neq0} a_{(n_1,n_2)}^{\theta^t\omega} \phi_{n_1}(z_i^t) \right]\ \left[\frac{1}{\nu_i k_i} \sum_{j\in\Nil_i\cap\Lil_{\delta}}\phi_{n_2}(x_{{s},j}^t) - \int_{\T} \phi_{n_2}(y) \mathrm{d}m(y)\right],
\end{align*}
which will be controlled by decay of Fourier coefficients, see Lemma \ref{lemma:zeta_i3}. Lastly, the low Fourier modes $\phi_{n_2}$ with $1\leq|n_2|\leq D$ make up the last term
\begin{align*}\tag*{(LFM)}
	\zeta_{i,\ell}(t,\omega,x,\omega_{{s}},x_{{s}}) :=& \alpha \frac{\nu_i k_i}{\Delta_0}\sum_{1\leq |n_2|\leq D} \left[a_{(0,n_2)}^{\theta^t\omega}   +  \sum_{n_1\neq0} a_{(n_1,n_2)}^{\theta^t\omega} \phi_{n_1}(z_i^t)   \right]\left[\frac{1}{\nu_i k_i} \sum_{j\in\Nil_i\cap\Lil_{\delta}}\phi_{n_2}(x_{{s},j}^t) - \int_{\T} \phi_{n_2}(y) \mathrm{d}m(y)\right],
\end{align*}
which will be controlled by rare visits of the shadowing orbit $\left(x_{{s},j}^t\right)_{j\in\Lil_{\delta}}$ to the bad sets, to be defined for each hub $i\in\Hil_{\Delta}$, see Lemma \ref{lemma:zeta_i4}.
In the next four lemmas, we estimate each of $\zeta_{i,d},\zeta_{i,c},\zeta_{i,h},\zeta_{i,\ell}$.


\begin{lemma}[Shadowing estimates for low degree nodes]\label{lemma:zeta_i1}
	Under the hypotheses of Theorem \ref{thm:loc_star}, we control the $\mathrm{(decoupling)}$ term
	\[ 
	|\zeta_{i,d}(t,\omega,x,\omega_{{s}},x_{{s}})|\leq \varepsilon/4,~~~~\forall i\in\Hil_{\Delta},~~\forall t\in\N.
	\]
\end{lemma}

\textit{Proof.}  
The estimate follows by assumption of  Theorem \ref{thm:loc_star} similar to the proof of Lemma \ref{lemma:zeta1}. \qed

\begin{lemma}[Non low degree nodes estimates]\label{lemma:zeta_i2}
	Under the hypotheses of Theorem \ref{thm:loc_star}, we control the contribution $\mathrm{(non ~ldn)}$ from non low degree neighbors
	\[
	|\zeta_{i,c}(t,\omega,x,\omega_{{s}},x_{{s}})|\leq \alpha(1-\nu)\sup_{t\in\N} \|h(\theta^t\omega,\cdot,\cdot)\|_{C^0},~~~~\forall i\in\Hil_{\Delta},~~\forall t\in\N.
	\]
\end{lemma}

\textit{Proof} follows directly from the $(\Delta,\delta,\nu)$-locally star-like properties of $G$. \qed 

\begin{lemma}[High Fourier modes controlled by decay of Fourier coefficients] \label{lemma:zeta_i3}
	Under the hypotheses of Theorem \ref{thm:loc_star}, by choosing
	\[
	D(\varepsilon,\omega) \asymp \varepsilon^{-1}\sup_{t\in\N }\|h(\theta^t\omega,\cdot,\cdot)\|_{C^4}
	\] 
	as in Lemma \ref{lemma:zeta2}, we estimate $\mathrm{(HFM)}$ 
	\[
	|\zeta_{i,h}(t,\omega,x,\omega_{{s}},x_{{s}})|\leq \varepsilon/2,~~~~\forall i\in\Hil_{\Delta},~~\forall t\in\N.
	\]
\end{lemma}

\textit{Proof.} 
The estimates are similar to Lemma \ref{lemma:zeta2} with $\alpha$ replaced by $\alpha_i = \alpha\frac{\nu_i\kappa_i}{\Delta_0}$.\qed

\begin{lemma}[Low Fourier modes controlled by rare visits to the bad sets]\label{lemma:zeta_i4}
	Under the hypotheses of Theorem \ref{thm:loc_star}, by choosing $$D(\varepsilon,\omega) \asymp \varepsilon^{-1} \sup_{t\in\N }\|h(\theta^t\omega,\cdot,\cdot)\|_{C^4},$$ we control $\mathrm{(LFM)}$ 
	\[
	|\zeta_{i,\ell}(t,\omega,x,\omega_{{s}},x_{{s}})|\leq \varepsilon/4,~~~~\forall i\in\Hil_{\Delta}
	\]
	with exponentially small exceptional asymptotic frequency at most $\rho$ with 
	$$\rho(\varepsilon,\omega)= 4MD(\varepsilon,\omega) \exp\left(-\nu\Delta\frac{450\varepsilon^2}{\alpha^2\sup_{t\in\N }\|h(\theta^t\omega,\cdot,\cdot)\|^2_{C^4}} \right).$$
\end{lemma}

\textit{Proof.} 
We estimate $|\zeta_{i,\ell}(t,\omega,x,\omega_{{s}},x_{{s}})|$ as a problem of frequency of visits to the bad sets 
$$B_{i,n}= B_i(\varepsilon_b, \phi_n):=\left\{x\in\T^N: \left|\frac{1}{\nu_i k_i}\sum_{j\in\Nil_i\cap\Lil_{\delta}} \phi_n(x_j) - \int_{\T} \phi_n(y)\mathrm{d}m(y)\right|>\varepsilon_b\right\},$$
for $i\in\Hil_{\Delta}$ and $n=\pm1,\cdots,\pm D$, totaling $2MD$ bad sets. By a similar argument to that for Size of Bad Sets Proposition \ref{prop:sizebadset}, we use Hoeffding inequality to obtain estimates for each of them
$$m^{\otimes N} (B_{i,n}) \leq 2\exp(-\nu_i k_i2^{-1} \|\phi_{n}\|_{C^0}^{-2} \varepsilon_b^2) \leq 2\exp(-\nu\Delta 2^{-1}\varepsilon_b^2).$$

An argument similar to Lemma \ref{lemma:RDSvisitfreq} implies that the shadowing orbit $x_{{s}}^t = (\varphi_{N-L+1}(t,\omega_{{s}}, x_{{s},N-L+1}),\cdots,\varphi_{N}(t,\omega_{{s}}, x_{{s},N}))$ visits each $B_{i,n}$ with exponentially small frequency at most 
$ 2\exp(-\nu\Delta 2^{-1} \varepsilon_b^2)$. 
Whenever all such bad sets are avoided, we estimate
\begin{align*}
	|\zeta_{i,\ell}(t,\omega,x,\omega_{{s}},x_{{s}})| \leq& \alpha \frac{\nu_i k_i}{\Delta_0}\sum_{1\leq |n_2|\leq D} \left[|a_{(0,n_2)}^{\theta^t\omega}|   +  \sum_{n_1\neq0} |a_{(n_1,n_2)}^{\theta^t\omega}|  \right]\left|\frac{1}{\nu_i k_i} \sum_{j\in\Nil_i\cap\Lil_{\delta}}\phi_{n_2}(x_{{s},j}^t) - \int_{\T} \phi_{n_2}(y) \mathrm{d}m(y)\right|\\
	\leq& \alpha \sum_{1\leq |n_2|\leq D} \left[ \frac{\sup_{t\in\N }\|h(\theta^t\omega,\cdot,\cdot)\|_{C^4}}{(2\pi)^4 n_2^4} + \sum_{n_1\neq 0} \frac{\sup_{t\in\N }\|h(\theta^t\omega,\cdot,\cdot)\|_{C^4}}{(2\pi)^4 n_1^2n_2^2} \right] \varepsilon_b\\
	\leq& \frac{\alpha\sup_{t\in\N }\|h(\theta^t\omega,\cdot,\cdot)\|_{C^4}}{(2\pi)^4 } \left[\frac{2\pi^4}{90} + \frac{2\pi^2}{6}\frac{2\pi^2}{6} \right]\varepsilon_b = \frac{\alpha\sup_{t\in\N }\|h(\theta^t\omega,\cdot,\cdot)\|_{C^4}}{120 } \varepsilon_b.
\end{align*}
By choosing 
$$\varepsilon_b=\frac{30}{\alpha\sup_{t\in\N }\|h(\theta^t\omega,\cdot,\cdot)\|_{C^4}}\varepsilon,~~~~D(\varepsilon,\omega)= \max\{D_1(\varepsilon,\omega),D_2(\varepsilon,\omega)\}\asymp \varepsilon^{-1} \sup_{t\in\N }\|h(\theta^t\omega,\cdot,\cdot)\|_{C^4},$$
we obtain $|\zeta_{i,\ell}(t,\omega,x,\omega_{{s}},x_{{s}})| \leq \varepsilon/4$ with exponentially small exceptional frequency at most $\rho$ with
\begin{align*}
	\rho(\varepsilon,\omega)=4MD(\varepsilon,\omega) \exp(-\nu\Delta 2^{-1} \varepsilon_b^2) =  4MD(\varepsilon,\omega) \exp\left(-\nu\Delta\frac{450\varepsilon^2}{\alpha^2\sup_{t\in\N }\|h(\theta^t\omega,\cdot,\cdot)\|^2_{C^4}}\right). \tag*{\qed}
\end{align*}

\textit{Proof of Theorem \ref{thm:loc_star}.} Fix, as in the hypotheses of Theorem \ref{thm:loc_star},
$$\varepsilon\geq \max\left\{\alpha \Delta^{-1/2}, 4\alpha \varepsilon_s \left(\alpha\Delta_0^{-1} \delta\sup_{t\in\N} \|h(\theta^t\omega,\cdot,\cdot)\|_{C^0}\right)\sup_{t\in\N }|h(\theta^t\omega,\cdot,\cdot)|_{\mathrm{Lip}}\right\},
$$
so that the (decoupling) term $\left|\zeta_{i,d}(t,\omega,x,{{s}},x_{{s}})\right|\leq \varepsilon/4$ for all $i\in\Hil_{\Delta}$ and $t\in\N$  by Lemma \ref{lemma:zeta_i1}. The locally star-like assumption of Theorem \ref{thm:loc_star} yields (non ldn) 
$|\zeta_{i,c}(t,\omega,x,\omega_{{s}},x_{{s}})|\leq \alpha(1-\nu)\sup_{t\in\N} \|h(\theta^t\omega,\cdot,\cdot)\|_{C^0}$ for all $i\in\Hil_{\Delta}$ and $t\in\N$ by Lemma \ref{lemma:zeta_i2}.  Now fix $D(\varepsilon,\omega) = \max\{D_1(\varepsilon,\omega),D_2(\varepsilon,\omega)\}$, so that (HFM) term $\left|\zeta_{i,h}(t,\omega,x,\omega_{{s}},x_{{s}})\right| \leq\varepsilon/2$ for all $i\in\Hil_{\Delta}$ and $t\in\N$  by Lemma \ref{lemma:zeta_i3} and (LFM) $\left|\zeta_{i,\ell}(t,\omega,x,\omega_{{s}},x_{{s}})\right| \leq\varepsilon/4$ for all $i\in\Hil_{\Delta}$ and with exponentially small exceptional asymptotic frequency  at most $\rho$ with 
$$\rho(\varepsilon,\omega)= 4MD(\varepsilon,\omega) \exp\left(-\nu\Delta\frac{450\varepsilon^2}{\alpha^2\sup_{t\in\N }\|h(\theta^t\omega,\cdot,\cdot)\|^2_{C^4}} \right).$$
by Lemma \ref{lemma:zeta_i4}. It is also with this exponentially small exceptional asymptotic frequency $\rho$ that we conclude
\begin{align*}
	\max_{i\in\Hil_{\Delta}}\left|\xi_{i,m}(\theta^t\omega,x^t)\right| \leq& 	\sum_{k=d,c,h,\ell}	\max_{i\in\Hil_{\Delta}}\left|\zeta_{i,k}(t,\omega,x,\omega_{{s}},x_{{s}})\right|
	\leq \varepsilon + \alpha(1-\nu)\sup_{t\in\N} \|h(\theta^t\omega,\cdot,\cdot)\|_{C^0}  \tag*{\qed}
\end{align*}

\section{Examples}\label{sec:eg_contractions}
In this section we provide a general class of examples of node dynamics, coupling functions and network structures satsifying the dimensional reduction Theorems \ref{thm:star} and \ref{thm:loc_star}.
\subsection{Node dynamics satisfying Shadowing and Typicality}
Consider the node dynamics given by the iid random iteration $\varphi$ of a measurable family of uniform contractions on a compact metric space $(M,d)$, endowed with its Borel $\sigma$-algebra $\Bil(M)$. 
More precisely, over a probability preserving transformation $(\Omega,\Fil,\PP,\theta)$, let
\begin{equation}\label{eq:iid_contractions}
	\varphi:\N\times\Omega\times M\to M
\end{equation}
satisfy:
\begin{enumerate}
	\item[(C1)] measurability with respect to $2^{\N}\otimes\Fil\otimes \Bil(M)$ and $\Bil(M)$;
	\item[(C2)] cocycle property over $\theta$: $\varphi(0,\omega,\cdot)=\mathrm{id}_M$ for all $\omega\in\Omega$ and $\varphi(t,\theta^s\omega,\cdot)\circ \varphi(s,\omega,\cdot)= \varphi(t+s,\omega,\cdot)$ for all $s,t\in\N$ and $\omega\in\Omega$;
	\item[(C3)] the maps $\varphi(1,\theta^t\omega,\cdot)$ have independent and identical distribution;
	\item[(C4)] there is a uniform contraction rate $\lambda\in(0,1)$ for which
	$$d(\varphi(1,\omega,x),\varphi(1,\omega,y))\leq \lambda d(x,y),~~~~\forall \omega\in\Omega, x,y\in M.$$
\end{enumerate}


Note that $\varphi$ is a continuous RDS over $\sigma$ in the sense of Arnold \cite{Arnold1998}.

\subsubsection{Unique stationary measure}
The iid random iteration $\varphi$ in (\ref{eq:iid_contractions}) induces, see \cite[Theorem 2.1.4]{Arnold1998}, the Markov chain
\begin{equation}\label{eq:iid_Markov_chain}
	Z_{n+1}(\omega):= \varphi(1,\theta^{n}\omega,Z_{n}(\omega)),~~~~n=0,1,\cdots,
\end{equation}
where $Z_0$ is independent of the random contractions $\varphi(1,\theta^t\omega,\cdot)$ with transition probability
$$P(x,B)=\PP(\omega\in\Omega: \varphi(1,\omega,x)\in B).$$

Note that the transition probability $P(x,B)$ acts on measures $\mu\in\Mil_1(M,\Bil(M))$ by
$$\mu\mapsto \mu P,~~~~\mu P(B):=\int_M P(x,B)\mathrm{d}\mu(x),~~\forall B\in\Bil(M),$$
and acts on bounded measurable functions $g\in b(M;\R)$ by
$$g\mapsto Pg,~~~~Pg(x):=\int_M g(y) P(x,\mathrm{d}y),~~~~\forall x\in M.$$
These two actions are in dual relation with each other
$$\int_M g \mathrm{d}\mu P = \int_M Pg \mathrm{d}\mu,~~~~\forall g\in b(M;\R),\mu\in\Mil_1(M,\Bil(M)).$$

A \textit{stationary measure} $m$ is defined by the property $mP=m$, or
$$\int_M P(x,B)\mathrm{d}m(x)=m(B),~~~~\forall B\in \Bil(M).$$

Define the \textit{coding map} by pre-composition
\begin{equation}\label{eq:iid_contractions_coding_map}
	\pi:\Omega\to M,~~~~\pi(\omega):=\lim_{n\to+\infty} \varphi(1,\omega,\cdot)\circ\varphi(1,\theta\omega,\cdot)\circ\cdots\circ \varphi(1,\theta^{n-1}\omega,p),~~p\in M,
\end{equation}
where the limit exists and is independent of the choice of initial point $p\in M$. For well-definedness of the coding map $\pi$, we have used the facts that $\varphi(1,\omega,\cdot)$ is a family of contractions with uniform rate $\lambda\in(0,1)$ and that $M$ is a compact metric space; note that the family $\{\varphi(1,\omega,\cdot):\omega\in\Omega\}$ of contractions can be finite, countable or uncountable. By Letac principle \cite{Letac1986}, the coding map implies the existence and uniqueness of stationary measure $m$ for the Markov chain (\ref{eq:iid_Markov_chain}).

\begin{prop} \label{prop:unique_stat_meas_iid_contractions}
	The Markov chain (\ref{eq:iid_Markov_chain}) admits a unique stationary measure $\pi_*\PP$ on $M$.
\end{prop}

\subsubsection{Typical random orbits}
Now we will establish a very strong sense of typicality for the random orbits of an iid random iteration of uniform contractions on a compact metric space.

\begin{prop}\label{prop:Breiman}
	Let $\varphi$ be an iid random iteration of continuous maps on compact metric state space $M$ and let $\{Z_n:n\in\N\}$ be the induced Markov chain. Then, with probability one, any accumulation point of the sequence $\frac{1}{N}\sum_{n=0}^{N-1} \delta_{Z_n}$ in $\Mil^1(M)$ endowed with the weak-star topology is a stationary measure.
\end{prop}

\begin{proof}
	In light of Breiman's Ergodic Theorem \cite{Breiman1960}, we only need to verify that $x\mapsto P(x,\cdot)$ is continuous in the weak-star topology on $\Mil^1(M)$. Since $M$ is a compact metric space, in particular, sequential, it suffices to check that $x_k\to x$ in $M$ implies $P(x_k,\cdot)\to P(x,\cdot)$ in $\Mil^1(M)$. For this, fix any $g\in C(M)\subseteq b(M,\Bil)$. It follows from RDS theory \cite[Theorem 2.1.4]{Arnold1998} that the transition probability $P$ is Feller, and in the particular case of compact state space $M$, we conclude $Pg$ is continuous. Now we have
	\begin{align*}
		P(x_k,g) = \int_M g(y) P(x_k,\mathrm{d}y) = (Pg)(x_k)\xrightarrow{k\to+\infty} (Pg)(x)=P(x,g)~~~~\text{because $Pg\in C(M)$},
	\end{align*}
	as required.
\end{proof}

\begin{theorem}[Uniform typicality of random orbits]\label{thm:typicality_randomcontractions}
	Consider the iid random iteration $\varphi$ of uniform contractions given in (\ref{eq:iid_contractions}) satisfying (C1-4).	There is a uniform set $\Omega_*\subseteq \Omega$ of full $\PP$-measure such that for every initial condition $x\in M$ and every noise $\omega\in\Omega_*$, the asymptotic behavior of the random orbit $\varphi(t,\omega,x)$ is described by the stationary measure
	$$\frac{1}{T} \sum_{t=0}^{T-1} \delta_{\varphi(t,\omega,x)}\xrightarrow[T\to+\infty]{\text{weak}^*} \pi_{*}\PP,$$
	where $\pi:\Omega\to M$ denotes the coding map given in (\ref{eq:iid_contractions_coding_map}).
\end{theorem}

\begin{rem}\label{rem:beyond_contractions}
	By assuming additionally that $\Omega$ is a compact metric space, \cite{Matias2020} obtains a similar type of uniform typicality for a class of iid random iteration of continuous maps strongly synchronizing on average. Also, by arguments similar to the one presented below, one can establish the uniform typicality for a class of so-called $J$-monotone maps on a compact connected subset of $\R^k$, see \cite[Theorem 2]{Matias2021}. 
\end{rem}

\begin{proof}
	Consider the Markov chain $\{Z_n\}$ in (\ref{eq:iid_Markov_chain}) 
	for some deterministic initial condition $x_0\in M$. 
	By uniqueness of stationary distribution $\pi_{*}\PP$ from Proposition \ref{prop:unique_stat_meas_iid_contractions}, we obtain from Proposition \ref{prop:Breiman} $\Omega_{x_0}\subseteq \Omega$ with $\PP(\Omega_{x_0})=1$ for which
	$$\frac{1}{T} \sum_{t=0}^{T-1} \delta_{\varphi(t,\omega,x_0)}= \frac{1}{T} \sum_{t=0}^{T-1} \delta_{Z_t(\omega)}\xrightarrow[T\to+\infty]{\text{weak}^*} \pi_{*}\PP,~~~~\forall \omega\in\Omega_{x_0}.$$
	Now fix any $\omega\in\Omega_*:=\Omega_{x_0}$, $x\in M$, $\varepsilon>0$, and $h\in C^0(M)$. We show that there is some $T_0$ for which
	$$T\geq T_0~~\Rightarrow~~\left|\left(\frac{1}{T} \sum_{t=0}^{T-1} \delta_{\varphi(t,\omega,x)}\right)(h) - \pi_{*}\PP(h)\right|\leq \varepsilon.$$
	Indeed,
	\begin{align*}
		&	\left|\left(\frac{1}{T} \sum_{t=0}^{T-1} \delta_{\varphi(t,\omega,x)}\right)(h) - \pi_{*}\PP(h)\right|\\
		\leq&  \left|\left(\frac{1}{T} \sum_{t=0}^{T-1} \delta_{\varphi(t,\omega,x)}\right)(h) - \left(\frac{1}{T} \sum_{t=0}^{T-1} \delta_{\varphi(t,\omega,x_0)}\right)(h) \right| + \left|\left(\frac{1}{T} \sum_{t=0}^{T-1} \delta_{\varphi(t,\omega,x_0)}\right)(h) - \pi_{*}\PP(h)\right|\\
		\leq& \frac{1}{T} \sum_{t=0}^{T-1} |h(\varphi(t,\omega,x)) - h(\varphi(t,\omega,x_0))| + \left|\left(\frac{1}{T} \sum_{t=0}^{T-1} \delta_{\varphi(t,\omega,x_0)}\right)(h) - \pi_{*}\PP(h)\right|
	\end{align*}
	By continuity and hence uniform continuity of $h$ on $M$, there is $\delta(\varepsilon)>0$ such that $d_M(x,y)\leq \delta(\varepsilon)$ implies $|h(x)-h(y)|\leq \varepsilon$. Also, since $\mathrm{diam} f_{\omega}^t(M) \leq \lambda^t \mathrm{diam} M\to 0$, it follows that there is some $T_1(\delta)$ such that $t\geq T_1(\delta)$ implies $\sup_{x,y\in M}d_M(\varphi(t,\omega,x),\varphi(t,\omega,y))\leq \delta$. For $T_2(\varepsilon)=T_1(\delta(\varepsilon/3))$, we have
	$$d_M(\varphi(t,\omega,x) , \varphi(t,\omega,x_0) )\leq \delta(\varepsilon/3),~~~~\forall t\geq T_2(\varepsilon)$$
	and hence 
	$$|h(\varphi(t,\omega,x)) - h(\varphi(t,\omega,x_0)) |\leq \varepsilon/3.~~~~\forall t\geq T_2(\varepsilon).$$
	Choose $T_3(\varepsilon)$ for which
	$$\frac{1}{T}T_2(\varepsilon)2\|h\|_{C^0}\leq \varepsilon/3,~~~~\forall T\geq T_3(\varepsilon).$$
	Combining the above two estimates, we obtain
	\begin{align*}
		&\frac{1}{T} \sum_{t=0}^{T-1} |h(\varphi(t,\omega,x)) - h(\varphi(t,\omega,x_0))|\\ \leq& \frac{1}{T}\sum_{t=0}^{T_2(\varepsilon)-1}  |h(\varphi(t,\omega,x)) - h(\varphi(t,\omega,x_0))|  + \frac{1}{T}\sum_{t=T_2(\varepsilon)}^{T-1}  |h(\varphi(t,\omega,x)) - h(\varphi(t,\omega,x_0))| \\
		\leq &\frac{T_2(\varepsilon) 2 \|h\|_{C^0} + (T-T_2(\varepsilon))\varepsilon/3}{T} \leq 2\varepsilon/3,~~~~\forall T\geq \max\{T_2(\varepsilon),T_3(\varepsilon)\}.
	\end{align*}

	By weak-star convergence $\frac{1}{T} \sum_{t=0}^{T-1} \delta_{\varphi(t,\omega,x_0)}\to \pi_{*}\PP$, there is some $T_4(\varepsilon)$ for which
	$$\left|\left(\frac{1}{T} \sum_{t=0}^{T-1} \delta_{\varphi(t,\omega,x_0)}\right)(h) - \pi_{*}\PP(h)\right|\leq \varepsilon/3,~~~~\forall T\geq T_4(\varepsilon).$$
	
	We thus conclude that for $T\geq T_0:=\max\{T_2(\varepsilon),T_3(\varepsilon),T_4(\varepsilon)\}$, we have
	$$\left|\left(\frac{1}{T} \sum_{t=0}^{T-1} \delta_{\varphi(t,\omega,x)}\right)(h) - \pi_{*}\PP(h)\right|\leq \varepsilon,$$
	as required.
\end{proof}

\subsubsection{Shadowing and Typicality on the star}

\begin{corollary}\label{cor:star_contractions}
	Consider star network dynamics (\ref{eq:hub}-\ref{eq:ldn}) satisfying (R1--3), where the node dynamics are given by iid random iteration of contractions $\varphi_i$ in (\ref{eq:iid_contractions}) with uniform contraction rate $\lambda\in(0,1)$ and unique stationary measure $m_0:=\pi_*\PP$ as in Theorem \ref{thm:typicality_randomcontractions}.
	Then, for any error tolerance 
	$$\varepsilon\geq \left\{\alpha L^{-1/2},4\alpha^2 L^{-1} (1-\lambda)^{-1}\sup_{t\in\N} \|h(\theta^t\omega,\cdot,\cdot)\|_{C^0} \sup_{t\in\N } |h(\theta^t\omega,\cdot,\cdot)|_{\mathrm{Lip}}\right\},$$
	the hub behavior admits $\varepsilon$-reduction to $\varphi_{\alpha,m_0}$ defined in (\ref{eq:reducedhub}), for almost every noise realization $\omega$ and any initial condition $x\in\T^N$ with exponentially small exceptional asymptotic frequency at most $\rho$ with
	$$\rho(\varepsilon,\omega)= D(\varepsilon,\omega)\exp(-L\varepsilon^2 \alpha^{-2}c(\omega)),$$
	where $D(\varepsilon,\omega)$ and $c(\omega)$ are the same constants independent of $N$ and $x$ as in Theorem \ref{thm:star}.
\end{corollary}

\begin{proof}
By (R1) and (C4), the $\varphi_i$ are independent random uniform contractions on $\T$ with rate $\lambda$. It follows that  the uncoupled system $\Phi_{\alpha=0}$ is a random uniform contraction satisfying (C4) on $\T^{N}$, equipped with the distance $d_{\T^{N}}(x,y)= \max_{i=1,\cdots,N} d_{\T}(x_i,y_i)$, with the same rate $\lambda$. By Theorem \ref{thm:typicality_randomcontractions} applied to $\Phi_{\alpha=0}$ on $(\T^{N},d_{\T^{N}})$, the random orbit $\Phi_{\alpha=0}(t,\omega,x)$ starting from any $x\in\T^N$ and $\omega\in\Omega_*$ asymptotically distributes as the unique stationary measure $m_0^{\otimes N}$.

To verify the (Typicality) assumption in Theorem \ref{thm:star}, simply take $x_s=x\in\T^N$ and $\omega_s=\omega\in\Omega_*$.
	For the (Shadowing) assumption, we compute
	\begin{align*}
		d_{\T}(x_j^t, \varphi_j(t,\omega,x_j)) \leq & d_{\T}\left(\varphi_j(1,\theta^{t-1}\omega,x_j^{t-1}) + \frac{\alpha}{L}  h(\theta^{t-1}\omega,x_j^{t-1},z^{t-1}), \varphi_j(1,\theta^{t-1}\omega)\circ \varphi_j(t-1,\omega,x_j)\right)\\
		\leq & \frac{\alpha}{L} \sup_{t\in\N} \|h(\theta^t\omega,\cdot,\cdot)\|_{C^0}+ \lambda\cdot d_{\T}(x_j^{t-1}, \varphi_j(t-1,\omega,x_j))\\
		\leq& \alpha  L^{-1} \sup_{t\in\N} \|h(\theta^t\omega,\cdot,\cdot)\|_{C^0}(1+\lambda + \cdot + \lambda^{t-1}) \\
		\leq & \alpha  L^{-1} \sup_{t\in\N} \|h(\theta^t\omega,\cdot,\cdot)\|_{C^0} \frac{1}{1-\lambda},~~~~\forall t\in\N.
	\end{align*}
	We have thus verified for almost every noise realization $\omega$ and for each initial condition $x\in\T^N$ (Shadowing) with 
	shadowing precision $\varepsilon_s =  \frac{1}{1-\lambda} \alpha L^{-1} \sup_{t\in\N} \|h(\theta^t\omega,\cdot,\cdot)\|_{C^0}$. 
	The corollary follows from Theorem \ref{thm:star}.
\end{proof}

\subsection{A locally star-like random power-law graph model} \label{sec:eg_loca_star_network}

We consider the expected degree sequence $\brw$ defined in (\ref{eq:ChungLu_PL_w_sequence}). 
Then the actual and expected degrees are close:
\begin{lemma}[Degree Concentration in $n$; \cite{Chung_2006} Lemma 5.7] \label{thm:CLlemma5.7}
	For a graph $G$ in $G(\brw)$, with probability $1-n^{-1/5}$ all nodes $i$ simultaneously satisfy
	$$|k_i- w_i| \leq 2 (\sqrt{w_i \log n} + \log n).$$
\end{lemma}

The Chung-Lu random power-law graph model 
is defined by (\ref{eq:ChungLu_PL_w_sequence}) without any canonical scale choices for hub and low degree scales $\Delta,\delta$. Therefore, we will first make more concrete choices of the parameters $m,w$, and then introduce the locally star-like parameters $(\Delta,\delta,\nu)$. 

\begin{theorem}[A locally star-like random power-law graph]\label{thm:loc_starlike_PL}
	Fix power-law exponent $\beta>2$. Consider a graph $G$ in the Chung-Lu random power-law graph model defined by the expected degree sequence $\brw$ in (\ref{eq:ChungLu_PL_w_sequence}) with 
	\begin{align*}
		m=& c_{\mathrm{hub}}\cdot n^{\frac{1}{\beta-1}},~~~~\text{for some }c_{\mathrm{hub}}\asymp 1,\\
		w=& o\left(n^{\frac{\lambda_{\mathrm{ldn}}}{\beta-1}}\right),~~~~\text{for some }\lambda_{\mathrm{ldn}}\in(0,1),
	\end{align*}
	and hub and low degree scales 
	$$\Delta\sim \lambda_{\mathrm{hub}} m, ~~\text{for some } \lambda_{\mathrm{hub}}\in(0,1);~~~~\delta\sim c_{\mathrm{ldn}} m^{\lambda_{\mathrm{ldn}}},~~\text{for some }c_{\mathrm{ldn}}\asymp1.$$ Then, with probability $1-O(n^{-1/5})$, we have
	$$M\asymp w^{\beta-1},~~~~L\sim n,~~~~\nu=1-O\left(n^{-\lambda_{\mathrm{ldn}}\frac{\beta-2}{\beta-1}} w^{\beta-2}\right),$$
	where
	\begin{enumerate}
		\item[(i)] the first $M$ nodes $1,\cdots,M$ are hubs of degree at least $\Delta$, 
		\item[(ii)] the last $L$ nodes $n-L+1,\cdots,n$ are low degree nodes of degree at most $\delta$,
		\item[(iii)] $G$ is $(\Delta,\delta,\nu)$-locally star-like. 
	\end{enumerate}
\end{theorem}

\begin{defn}[Locally star-like random power-law graph model]
	We will call the model considered in Theorem \ref{thm:loc_starlike_PL} the \textit{locally star-like random power-law graph model} and denote it by $\mathrm{LSL}(\beta,n)$.
\end{defn}

\textit{Proof of Theorem \ref{thm:loc_starlike_PL}.}
Our strategy is to count the number of nodes that are expected to be hubs in Step \textbf{I} and low degree nodes in Step \textbf{II}, prove the expected picture satisfies the locally star-like property in Step \textbf{III}, and finally bring the estimates to the actual degrees by concentration.

 \textbf{I.} First we compute the number $M$ of expected hubs by solving the equation $w_{M}=\lambda_{\mathrm{hub}}\cdot m$, which reads
\begin{align*}
	\frac{\beta-2}{\beta-1} w n ^{\frac{1}{\beta-1}} \cdot \left(n\left(\frac{w(\beta-2)}{m(\beta-1)}\right)^{\beta-1} +M -1\right)^{-\frac{1}{\beta-1}}=& \lambda_{\mathrm{hub}}\cdot m.
\end{align*}
By multiplying $\left(\frac{\beta-2}{\beta-1} w n ^{\frac{1}{\beta-1}} \right)^{-1}$ and taking the $-(\beta-1)$-power on both sides, we obtain
\begin{align*}
	M 
	=&1  + n \left(\frac{w(\beta-2)}{m(\beta-1)}\right)^{\beta-1} (\lambda_{\mathrm{hub}}^{-(\beta-1)} - 1).
\end{align*}
By the choices of $w$ and $m$, we obtain
\begin{align*}
	M\sim & n w^{\beta-1} (c_{\mathrm{hub}} n^{\frac{1}{\beta-1}})^{-(\beta-1)} \left(\frac{\beta-2}{\beta-1}\right)^{\beta-1}  (\lambda_{\mathrm{hub}}^{-(\beta-1)} - 1) = w^{\beta-1} c_{\mathrm{hub}}^{-(\beta-1)}\left(\frac{\beta-2}{\beta-1}\right)^{\beta-1} (\lambda_{\mathrm{hub}}^{-(\beta-1)} - 1) \asymp w^{\beta-1}.
\end{align*}

\textbf{II.} Then we compute $L$ by solving the equation $w_{n-L+1}=c_{\mathrm{ldn}}\cdot m^{\lambda_{\mathrm{ldn}}}$, which by definition reads
\begin{align*}
	\frac{\beta-2}{\beta-1} w n ^{\frac{1}{\beta-1}} \cdot (n-L+1)^{-\frac{1}{\beta-1}}=& c_{\mathrm{ldn}}\cdot m^{\lambda_{\mathrm{ldn}}}.
\end{align*}
By multiplying $\left(\frac{\beta-2}{\beta-1} w n ^{\frac{1}{\beta-1}} \right)^{-1}$ and taking $-(\beta-1)$-power on both sides, we obtain
\begin{align*}
	n-L+1 = & \left(\frac{\frac{\beta-2}{\beta-1} w n ^{\frac{1}{\beta-1}} }{c_{\mathrm{ldn}}\cdot m^{\lambda_{\mathrm{ldn}}}}\right)^{\beta-1}\\
	L=&  n \left[ 1 + n^{-1}- m^{-(\beta-1)\lambda_{\mathrm{ldn}}} w^{\beta-1} c_{\mathrm{ldn}}^{-(\beta-1)} \left(\frac{\beta-2}{\beta-1}\right)^{\beta-1}\right].
\end{align*}
By choices of $w$ and $m$, we continue
\begin{align*}
	L\sim& n\left[1-  n^{-\lambda_{\mathrm{ldn}}} c_{\mathrm{hub}}^{-(\beta-1)\lambda_{\mathrm{ldn}}} w^{\beta-1}c_{\mathrm{ldn}}^{-(\beta-1)} \left(\frac{\beta-2}{\beta-1}\right)^{\beta-1}\right]\\
	L=& n\left[1-  O(n^{-\lambda_{\mathrm{ldn}}}w^{\beta-1}) \right]\sim n.
\end{align*}

\textbf{III.} To count the number of expected low degree nodes among neighbors of an expected hub $i=1,\cdots,M$, we consider the random variable
$$L_i=\sum_{j=n-L+1}^n X_{ij}$$
which counts the number of neighbors of $i$ that are expected low degree nodes. 
As the edge experiments $X_{ij}$ are independent, Chernoff inequality \cite[Theorem 2.7]{Chung_2006} with $\lambda= C\sqrt{\E[L_i]}$ for a free parameter $C>0$ yields
\begin{equation}\label{eq:LSL_1}
	\PP(L_i \leq \E[L_i] - C\sqrt{\E[L_i]}) \leq \exp(-C^2/2).
\end{equation}

Now we prove an auxiliary lemma to estimate the mean expected degree.
\begin{lemma}\label{lemma:nw}
	Under the assumptions of Theorem \ref{thm:loc_starlike_PL}, we have
	$$\sum_{i=1}^{n}w_i \sim nw.$$
\end{lemma}
\textit{Proof of Lemma \ref{lemma:nw}.} Writing $i_0:=n \left(\frac{w(\beta-2)}{m(\beta-1)}\right)^{\beta-1}$, we compute
\begin{align*}
	\frac{1}{n}\sum_i w_i \sim& \frac{1}{n} \int_{i_0}^{i_0+n} c \cdot i^{-\frac{1}{\beta-1}} \mathrm{d}i = \frac{1}{n}\frac{\beta-2}{\beta-1} wn^{\frac{1}{\beta-1}} \left[\frac{i^{1-\frac{1}{\beta-1}}}{1-\frac{1}{\beta-1}}\right]^{i_0+n}_{i_0} =  wn^{\frac{1}{\beta-1}-1} \left[(i_0+n)^{\frac{\beta-2}{\beta-1}} - i_0^{\frac{\beta-2}{\beta-1}}\right]\\
	\sim & w. \tag*{\qed}
\end{align*}
Now we compute
\begin{align*}
	\E[L_i]= \sum_{j=i_0+n-L}^{i_0+n-1} \frac{w_i w_j}{\sum_k w_k} \sim& \frac{w_i}{wn} \int_{i_0+n-L}^{i_0+n} w_j\mathrm{d}j= \frac{w_i}{wn}\int_{i_0+n-L}^{i_0+n} \frac{\beta-2}{\beta-1} w n^{\frac{1}{\beta-1}} j^{-\frac{1}{\beta-1}} \mathrm{d}j \\
	=& \frac{w_i}{wn} \frac{\beta-2}{\beta-1} w n^{\frac{1}{\beta-1}} \left[ \frac{j^{1-\frac{1}{\beta-1}}}{1-\frac{1}{\beta-1}}\right]_{i_0+n-L}^{i_0+n}\\
	=& 
	w_i \left[(\varepsilon+1)^{\frac{\beta-2}{\beta-1}}- (\varepsilon+1-L/n)^{\frac{\beta-2}{\beta-1}}\right],
\end{align*}
where $\varepsilon:= \left(\frac{w(\beta-2)}{m(\beta-1)}\right)^{\beta-1}=O(w^{\beta-1} n^{-1})$. Using $L\sim n\left[1-  O(n^{-\lambda_{\mathrm{ldn}}}w^{\beta-1}) \right]$ from Step II, we continue
\begin{align*}
	\E[L_i] \sim w_i \left[1+\frac{\beta-2}{\beta-1}\varepsilon +O(\varepsilon^2)- \left( O(n^{-\lambda_{\mathrm{ldn}}}w^{\beta-1}) \right)^{\frac{\beta-2}{\beta-1}}\right] = w_i \left[1+O( n^{-\lambda_{\mathrm{ldn}}\frac{\beta-2}{\beta-1}} w^{\beta-2})\right].
\end{align*}

Since $w_i\geq w_{i_0+M-1} = \lambda_{\mathrm{hub}} \cdot m \gg \log n$, we put $C = 2\sqrt{\log n}$ in eq. (\ref{eq:LSL_1}) and obtain $\PP(L_i \leq \E[L_i] - 2\sqrt{\E[L_i]\log n}) \leq n^{-2}.$ Now using $\E[ k_i]=w_i=\mathrm{Var}[ k_i]$, $M=1$, and $\lambda=2\sqrt{w_i\log n}$, Chernoff upper tail bound \cite[Theorem 2.6]{Chung_2006} yields
\begin{align*}
	\PP( k_i\geq w_i + 2\sqrt{w_i\log n}) \leq& \exp\left(-\frac{4w_i \log n}{2(w_i + 2\sqrt{w_i \log n}/3)}\right) = \exp\left(-\frac{2 \log n}{1 + (2/3)\sqrt{ \log n / w_i})}\right) \\
	\leq& \exp\left(-\frac{2 \log n}{1 + (2/3)}\right)  = n^{-6/5},
\end{align*}
where the second inequality follows from our choice that the expected hub degree $w_i\geq \Delta>\log n$.

Hence, with probability $1-M(n^{-2}+n^{-6/5})$, we have simultaneously for each hub $i=i_0,\cdots,i_0+M-1$ that
\begin{align*}
	\frac{L_i}{ k_i} > \frac{\E[L_i] - 2\sqrt{\E[L_i]\log n}}{w_i + 2\sqrt{w_i\log n}} =1+O( n^{-\lambda_{\mathrm{ldn}}\frac{\beta-2}{\beta-1}} w^{\beta-2}) =: \nu.
\end{align*}

With probability $1-O(n^{-1/5})$, our random graph model is $(\Delta,\delta,\nu)$-locally star-like. \qed

\subsubsection{Shadowing and Typicality on a power-law network}
%
	
\begin{corollary}\label{cor:PL_contractions}
	Consider $G$-network dynamics (\ref{eq:Gnetwork_dynamics}) satisfying (R1--3), where $G$ is a $(\Delta,\delta,\nu)$-locally star-like network on $N\gg1$ nodes with largest degree $\Delta_0$ from the random power-law graph model as in Theorem \ref{thm:loc_starlike_PL}, and node dynamics given by iid random iteration of contractions $\varphi_i$ with uniform contraction rate $\lambda\in(0,1)$ and unique stationary measure $m_0$ as in Corollary \ref{cor:star_contractions}.
	Then, for any error tolerance 
	\[
	\varepsilon\geq \left\{\alpha\Delta^{-1/2},4\alpha^2 \delta\Delta_0^{-1} (1-\lambda)^{-1}\sup_{t\in\N} \|h(\theta^t\omega,\cdot,\cdot)\|_{C^0} \sup_{t\in\N } |h(\theta^t\omega,\cdot,\cdot)|_{\mathrm{Lip}}\right\},
	\]
	each hub $i\in\Hil_{\Delta}$ admits $(\varepsilon + \alpha(1-\nu)\sup_{t\in\N} \|h(\theta^t\omega,\cdot,\cdot)\|_{C^0})$-reduction to $\varphi_{\alpha_i,m_0}$ in Eq. (\ref{eq:reducedhub}) with $\alpha_i=\alpha\frac{\nu_i\kappa_i}{\Delta_0}$, for almost every noise realization $\omega$ and any initial condition $x\in\T^N$,
	with exceptional asymptotic frequency at most $\rho$ with  
	$$\rho(\varepsilon,\omega)= 4 MD(\varepsilon,\omega)\exp(-\nu\Delta\varepsilon^2 \alpha^{-2}c(\omega)),$$
	where $D(\varepsilon,\omega)$ and $c(\omega)$ are the same constants independent of $N$ and $x$ as in Theorem \ref{thm:star}.
\end{corollary}

\begin{proof}
As in the proof of Corollary \ref{cor:star_contractions}, we take $x_s=x\in\T^N$ and $\omega_s=\omega\in\Omega_*$, verifying the (Typicality) assumption in Theorem \ref{thm:loc_star}. For the (Shadowing) assumption, we compute

	\begin{align*}
		d_{\T}(x_j^t, \varphi_j(t,\omega,x_j)) \leq & d_{\T}(\varphi_j(1,\theta^{t-1}\omega,x_j^{t-1}) + \frac{\alpha}{\Delta_0} \sum_{j=1}^N A_{jk} h_{\theta^{t-1}\omega}(x_j^{t-1},x_k^{t-1}), \varphi_j(1,\theta^{t-1}\omega)\circ \varphi_j(t-1,\omega,x_j))\\
		\leq & \frac{\alpha}{\Delta_0} \delta \sup_{t\in\N} \|h(\theta^t\omega,\cdot,\cdot)\|_{C^0}+ \lambda\cdot d_{\T}(x_j^{t-1}, \varphi_j(t-1,\omega,x_j))\\
		\leq& \alpha \delta \Delta_0^{-1} \sup_{t\in\N} \|h(\theta^t\omega,\cdot,\cdot)\|_{C^0}(1+\lambda + \cdot + \lambda^{t-1}) \\
		\leq & \alpha \delta \Delta_0^{-1} \sup_{t\in\N} \|h(\theta^t\omega,\cdot,\cdot)\|_{C^0} \frac{1}{1-\lambda},~~~~\forall t\in\N.
	\end{align*}
	We have thus verified almost surely and for each initial condition $x\in\T^N$ the (Shadowing) assumption with 
	shadowing precision $\varepsilon_s =  \frac{1}{1-\lambda} \alpha\Delta_0^{-1} \delta\sup_{t\in\N} \|h(\theta^t\omega,\cdot,\cdot)\|_{C^0}$. 
	The corollary follows from Theorem \ref{thm:loc_star}.
\end{proof}


\appendix
\section{Proofs of Theorems A, C and more general settings}
\label{appendix:ac_shad}

In addition to the hypotheses of Theorems \ref{thm:star} and \ref{thm:loc_star}, if $\PP\otimes m^{\otimes N}$-a.e. $(\omega,x)\in\Omega\times \T^N$ admits shadowing intial data $(\omega_{{s}},x_{{s}})$ satisfying (Shadowing) and (Typicality), then the $(m,\varepsilon)$-reduction holds $\PP\otimes m^{\otimes N}$-almost surely. 
Moreover, if shadowing can be realized in an absolutely continuous way, then one can transfer the measure-theoretic results (ii) Small fluctuation in long time windows and (iii) Gaussian fluctuations from decoupled low degree node dynamics to the coupled network system. 

\begin{defn}[Absolutely continuous shadowing]
	We say that a network system (\ref{eq:Gnetwork_dynamics}) admits \textit{absolutely continous shadowing} if there is a $\PP\otimes m^{\otimes N}$-invariant map $\mathrm{ACS}:(\omega,x)\mapsto (\omega_{{s}},x_{{s}})$ for which the network trajectory $x^t$ starting from $\PP\otimes m^{\otimes N}$-almost every $(\omega,x)$ admits shadowing initial data $(\omega_{{s}},x_{{s}})$ satisfying (Shadowing) and (Typicality) as in Theorem \ref{thm:loc_star}.
\end{defn}


We state and prove the result for the star network, assuming $\sup_{\omega\in\Omega} \|h(\omega,\cdot,\cdot)\|_{C^4}\leq K$ for some constant $K>0$.

\begin{theorem}\label{thm:star_ACS}
	In addition to the hypotheses of Theorem \ref{thm:star}, assume also that the star network dynamics admits absolutely continuous shadowing map $\mathrm{ACS}$ and $\sup_{\omega\in\Omega} \|h(\omega,\cdot,\cdot)\|_{C^4}\leq K$ for some constant $K>0$. Then, we have the following.
	\begin{enumerate}
		\item[$\mathrm{(i)}$] \textbf{Almost sure reduction in asymptotic time:} for any fixed error tolerance $$\varepsilon\geq \max\left\{\alpha L^{-1/2}, 4\alpha K\varepsilon_s (\alpha L^{-1} K)\right\},$$
		starting from $\PP\otimes m^{\otimes N}$-almost every $(\omega,x)\in\Omega\times \T^N$, the hub behavior (\ref{eq:hub}) admits $\varepsilon$-reduction to $\varphi_{\alpha,m}$ with exceptional asymptotic frequency at most $\rho$ with
		$$ \rho\asymp \varepsilon^{-1} \exp(-L\varepsilon^2 \alpha^{-2}).$$ 
		\item[$\mathrm{(ii)}$] \textbf{Small fluctuation in long time windows:} there is an exceptional set $E$ of initial data in $\Omega\times \T^{N}$ of size 
		$$\PP\otimes m^{\otimes N}(E) \asymp L^{\kappa} \exp(-L^{1-2\kappa}),~~~~\kappa\in(0,1/2),$$ outside of which
		the fluctuation is well-controlled in time windows $I_{t_0}^T:=\left\{t_0,\cdots,t_0+T-1\right\}$ of exponential length $T\geq \exp( L^{1-2\kappa})$ starting at any moment $t_0\in\N$ in the sense that 
		$$\max_{t\in I_{t_0}^T} |\xi_{m}(\theta^t\omega, \Phi_{\alpha}(t,\omega,x))| \leq L^{-\kappa} \alpha K^2,~~~~\forall t_0\in\N,\forall (\omega,x)\in \Omega\times \T^N\setminus E;$$
		\item[$\mathrm{(iii)}$] \textbf{Gaussian fluctuations: } if the coupling map has the form $h(\omega,x,y)=\phi(x)-\phi(y)+\psi(\omega)$ for some Lipschitz $\phi:\T\to\R$ and bounded $\psi$, then at any time $t\in\N$, the fluctuation $\xi_{m}(\theta^t\omega, \Phi_{\alpha}(t,\omega,x))$ is approximately Gaussian, i.e.,
		$$\forall s\in\R:~~~~ \PP\otimes m^{\otimes N}\left\{(\omega,x): \xi_{m}(\theta^t\omega, \Phi_{\alpha}(t,\omega,x)) \leq s\right\} \in\left[ F_L\left( s -c_1\right)-c_2, F_L\left( s+c_1 \right) +c_2\right],$$
		where $c_1\asymp \varepsilon_s \left(\alpha L^{-1} \left( 2\|\phi\|_{C^0}+ \|\psi\|_{L^{\infty}} \right) \right)$, $c_2\asymp L^{-1/2}$, and	$F_L$ denotes the cdf of the normal distribution with zero mean and variance $\alpha^2 \left[\int_{\T}\phi^2\mathrm{d}m - \left(\int_{\T}\phi\mathrm{d}m\right)^2\right] $. 
	\end{enumerate}
\end{theorem}


Item \textbf{(i) Almost sure reduction in asymptotic time} of Theorem \ref{thm:star_ACS} can be proven using Corollaries \ref{cor:star_contractions} and \ref{cor:PL_contractions}. The particular constants are simpler and sharper in the Theorem than in the Corollaries because the coupling function in the introduction example has only one Fourier mode and hence requires no truncation; these detailed calculations are presented in full in \cite{BianPhD}.

\textit{Proof of Theorem \ref{thm:star_ACS} (ii).}  Define
\begin{align*}
	Q_{{s}} (t,\phi,\varepsilon_b):= \left\{(\omega_{{s}}, x_{{s}}): \left|\frac{1}{L}\sum_{j=2}^N \phi\circ \varphi_j(t,\omega_{{s}},x_{{s},j}) - \int_{\T} \phi\mathrm{d}m\right| >\varepsilon_b \right\}.
\end{align*}
Note $Q_{{s}} (0,\phi,\varepsilon_b)=\Omega\times B(\varepsilon_b,\phi)$ and hence $m^{\otimes N}(B(\varepsilon_b,\phi))=\PP\otimes m^{\otimes N} (Q_{{s}}(0,\phi,\varepsilon_b))$. Since
$m$ is stationary, it follows that $\PP\otimes m^{\otimes L}$ is invariant for the $L$-fold skew product $$(\omega_{{s}},x_{{s},2},\cdots,x_{{s},N}) \mapsto \left(\theta^t\omega_{{s}}, \varphi_2(t,\omega_{{s}},x_{{s},2}),\cdots, \varphi_N(t,\omega_{{s}},x_{{s},N})\right),$$
and thus
$m^{\otimes N}(B(\varepsilon_b,\phi))=\PP\otimes m^{\otimes N} (Q_{{s}}(t,\phi,\varepsilon_b))$ for all $t\in\N$, where the hub coordinate is free. By Hoeffding, we obtain
$$\PP\otimes m^{\otimes N} (Q_{{s}}(t,\phi,\varepsilon_b)) = m^{\otimes N}(B(\varepsilon_b,\phi))\leq 2\exp\left(-L 2^{-1} \|\phi\|_{\infty}^{-2} \varepsilon_b^2\right).$$
For $\phi_n(x):=e^{2\pi i nx}$ we have
$$\PP\otimes m^{\otimes N} (Q_{{s}}(t,\phi_n,\varepsilon_b)) \leq 2\exp\left(-L 2^{-1} \varepsilon_b^2\right),$$
and therefore, in the notation adopted in the proof ot Theorem \ref{thm:star}, we have
\begin{align*}
	4TD\exp(-L2^{-1}\varepsilon_b^2)\geq &\PP\otimes m^{\otimes N} \left(\bigcup_{t\in I_{t_0}^T} \bigcup_{1\leq |n_2|\leq D} Q_{{s}}(t,\phi_{n_2},\varepsilon_b)\right)=\PP\otimes m^{\otimes N} \left(\mathrm{ACS}^{-1}\bigcup_{t\in I_{t_0}^T} \bigcup_{1\leq |n_2|\leq D} Q_{{s}}(t,\phi_{n_2},\varepsilon_b)\right).
\end{align*}
For any $(\omega,x)\notin \mathrm{ACS}^{-1}\bigcup_{t\in I_{t_0}^T} \bigcup_{1\leq |n_2|\leq D} Q_{{s}}(t,\phi_{n_2},\varepsilon_b)$, we have
\begin{align*}
	\left|\frac{1}{L}\sum_{j=2}^N \phi_{n_2}\circ \varphi_j(t,\omega_{{s}},x_{{s},j}) - \int_{\T} \phi_{n_2}\mathrm{d}m\right| \leq \varepsilon_b,~~~~\forall t\in I_{t_0}^T,1\leq|n_2|\leq D,
\end{align*}
where $(\omega_{{s}},x_{{s}})=\mathrm{ACS}(\omega,x)$, and hence, by choosing $\varepsilon_b=\frac{30}{\alpha K}\varepsilon$, we obtain
$$|\zeta_{\ell}(t,\omega,x,\omega_{{s}},x_{{s}})|\leq \varepsilon/4,~~~~\forall t\in I_{t_0}^T.$$
In summary, for any fixed $$\varepsilon\geq 4\alpha K \varepsilon_s (\alpha L^{-1} K),~~~~D(\varepsilon)\asymp \varepsilon^{-1},~~~~\varepsilon_b=\frac{30}{\alpha K}\varepsilon,$$ 
where $D(\varepsilon)=D(\varepsilon,\omega)$ is now independent of $\omega$ by assumption $\sup_{\omega} \|h(\omega,\cdot,\cdot)\|_{C^4}\leq K$, we have
$$|\xi_{m}(\theta^t\omega,x^t)| \leq \sum_{k=d,h,\ell}|\zeta_k(t,\omega,x,\omega_{{s}},x_{{s}})| \leq \varepsilon/4 + \varepsilon/2 + \varepsilon/4=\varepsilon,$$
for any $(\omega,x)\notin \mathrm{ACS}^{-1}\bigcup_{t\in I_{t_0}^T} \bigcup_{1\leq |n_2|\leq D} Q_{{s}}(t,\phi_{n_2},\varepsilon_b)$. 

Now take $\varepsilon= L^{-\kappa}\alpha K^2$ for some $\kappa\in(0,1/2)$ and $T=\exp\left(L \frac{\varepsilon^2}{\alpha^2 K^2}\right)= \exp(L^{1-2\kappa})$, we have 
$$	\PP\otimes m^{\otimes N} \left(\mathrm{ACS}^{-1}\bigcup_{t\in I_{t_0}^T} \bigcup_{1\leq |n_2|\leq D} Q_{{s}}(t,\phi_{n_2},\varepsilon_b)\right)\leq 4TD\exp(-L2^{-1}\varepsilon_b^2)\asymp L^{\kappa} \exp(-L^{1-2\kappa}).$$
This completes the proof of (ii) Small fluctuation in long time windows. \qed

\noindent \textit{Proof of Theorem \ref{thm:star_ACS} (iii).} 
Fix any time moment $t\in\N$. The fluctuation has the form
\[
\xi^t = \frac{\alpha}{L} \sum_{j=2}^N  \left[ \phi(x_i^t) -\int_{\T}\phi\mathrm{d}m \right].
\]
We compare the low degree trajectory $x_i^t$ observed through $\phi$ with the iid random variables
$$Y_i^t:\Omega\times \T^N\to\R,~~~~Y_i^t(\omega_s,x_s):= \alpha\left[ \phi \circ \mathrm{proj}_i\circ \varphi_i(t,\omega_s,x_s) - \int_{\T}\phi\mathrm{d}m\right],~~~~i=2,\cdots,N.$$
Since $\phi$ is Lipschitz, we obtain that $Y_i^t$ are bounded (from mean value theorem)
\[
|Y_i^t| \leq \alpha |\phi|_{\mathrm{Lip}},
\]
zero mean
\begin{align*}
	\mu:=&\E[Y_i^t]=\alpha\left[\int_{\Omega\times\T^N} \phi\circ \mathrm{proj}_i\circ \varphi_i(t,\omega_s,x_s)\mathrm{d}\PP\otimes m^{\otimes N}(\omega_s,x_s) - \int_{\T}\phi\mathrm{d}m\right]\\
	=&\alpha\left[\int_{\Omega\times\T^N} \phi\circ \mathrm{proj}_i \mathrm{d}\PP\otimes m^{\otimes N}(\omega_s,x_s) - \int_{\T}\phi\mathrm{d}m\right] =0,
\end{align*}
finite variance
$$\sigma^2:=\E[|Y_i^t|^2] = \alpha^2 \left[\int_{\T}\phi^2\mathrm{d}m - \left(\int_{\T}\phi\mathrm{d}m\right)^2\right] \leq 4|\phi|_{\mathrm{Lip}}^2\alpha^2,$$
and finite third moment
$$\tau:=\E[|Y_i^t|^3] \leq 8|\phi|_{\mathrm{Lip}}^3\alpha^3.$$
Now by shadowing estimates and absolutely continuous shadowing map $\mathrm{ACS}$, we have
\begin{align*}
	\PP\otimes m^{\otimes N}\left\{\xi^t \leq s\right\} =&\mathrm{ACS}_*\PP\otimes m^{\otimes N}\left\{(\omega,x): \frac{1}{L} \sum_{j=2}^N  \alpha\left[ \phi(x_i^t) -\int_{\T}\phi\mathrm{d}m \right] \leq s\right\} \\
	\leq& \PP\otimes m^{\otimes N}\left\{(\omega_s,x_s): \frac{1}{L} \sum_{j=2}^N  Y_i^t\leq s + |\phi|_{\mathrm{Lip}} \alpha\varepsilon_s\right\} \\
	=& \PP\otimes m^{\otimes N}\left\{(\omega_s,x_s): \frac{1}{\sigma \sqrt{L}} \sum_{j=2}^N  Y_i^t\leq (s + |\phi|_{\mathrm{Lip}} \alpha\varepsilon_s) \frac{\sqrt{L}}{\sigma}\right\} \\
	\leq & \Phi_{\Nil(0,1)}\left((s+ |\phi|_{\mathrm{Lip}}\alpha\varepsilon_s )\frac{\sqrt{L}}{\sigma}\right) + \frac{3\tau}{\sigma^3}\sqrt{L} \\
	=& \Phi_{\Nil(0,\sigma^2/L)}\left(s+ |\phi|_{\mathrm{Lip}}\alpha\varepsilon_s \right) + \frac{3\tau}{\sigma^3\sqrt{L} }
\end{align*}
where the second approximation follows by Berry-Esseen \cite[Theorem 3.4.9]{Durrett2019}, the last equality follows from the Gaussian cdf relation $\Phi_{\Nil(\mu,\sigma^2)}(s) = \Phi_{\Nil(0,1)}\left(\frac{x-\mu}{\sigma}\right)$, and $\sigma^2 =\alpha^2 \left[\int_{\T}\phi^2\mathrm{d}m - \left(\int_{\T}\phi\mathrm{d}m\right)^2\right] $, $\varepsilon_s = \varepsilon_s \left(\alpha L^{-1} \left( 2\|\phi\|_{C^0}+ \|\psi\|_{L^{\infty}} \right) \right)$. The lower bound can be obtained similarly.
\qed

\subsection{Proof of Theorem A} \label{sec:proof_theorem_A}
Since the coupling function (\ref{eq:h_omega}) has only one Fourier mode, the Fourier truncation in Lemma \ref{lemma:zeta2} is unnecesary and the bad set estimates in Lemma \ref{lemma:zeta3} reduce to one bad set only. 

Item (i) follows by a simplified version of Theorem \ref{thm:star} together with the uniform typicality Theorem \ref{thm:typicality_randomcontractions}. The shadowing precision $\varepsilon_s=\frac{17\alpha}{3L}$ can be calculated as in Corollary \ref{cor:star_contractions}. We choose $\varepsilon_b=\frac{\varepsilon}{2\alpha}$ so that the shadowing and bad set arguments each contribute $\varepsilon/2$ to the reduction estimate.

Item (ii) follows from a similar argument to Theorem \ref{thm:star_ACS} Item (ii) with $\varepsilon_b=L^{-\kappa}\sqrt{6}$ and $\varepsilon_s=\frac{17\alpha}{3L}$.

Item (iii) follows from Theorem \ref{thm:star_ACS} Item (iii) with $\phi(x)=\sin 2\pi x$, $\psi(\omega)=\omega/3.6$.

\subsection{Proof of Theorem C}\label{sec:proof_theorem_C}
Item (i) follows by a simplified version of Theorem \ref{thm:loc_star} together with the uniform typicality Theorem \ref{thm:typicality_randomcontractions}. The non ldn contribution from Lemma \ref{lemma:zeta_i2} is $|\zeta_{i,c}|\leq \frac{17}{6}\alpha(1-\nu)\leq \varepsilon/3$, by our choice of $\varepsilon$. The shadowing precision $\varepsilon_s=\frac{17\alpha \delta}{3\Delta}$ can be calculated as in Corollary \ref{cor:PL_contractions}. We choose $\varepsilon_b=\frac{\varepsilon}{3\alpha}$ so that the shadowing and bad set arguments each contribute $\varepsilon/3$ to the reduction estimate.

Item (ii) follows from a similar argument to Theorem \ref{thm:star_ACS} Item (ii) with $\varepsilon_b=\Delta^{-\kappa}\sqrt{6}$ and $\varepsilon_s=\frac{17\alpha\delta}{3\Delta}$.

Item (iii) follows from a similar argument to Theorem \ref{thm:star_ACS} Item (iii) with $\phi(x)=\sin 2\pi x$, $\psi(\omega)=\omega/3.6$. The correction constants $s_1,s_2$ need to account for non ldn contribution.
$$s_1 : = \alpha_i \frac{1}{\nu_i k_i}\sum_{j\in\Nil_i\cap\Lil_{\delta}} (\sin2\pi x_i^t-\sin 2\pi y_j^t ) + \alpha \frac{1}{\Delta_0} \sum_{j\in\Nil\setminus \Lil_{\delta}} h_{\omega_i^t}(z_i^t,x_j^t).$$
Using $|h|\leq 2+ \frac{3}{3.6} = \frac{17}{6}$ and $\varepsilon_s  = \frac{17}{3}\alpha\frac{\delta}{\Delta}$, we have
$$|s_1| \leq \alpha_i 2\pi \varepsilon_s  + \alpha \Delta_0^{-1} (1-\nu_i) k_i \frac{17}{6} = \alpha_i 2\pi \frac{17}{3}\alpha\frac{\delta}{\Delta} + \alpha \Delta_0^{-1} (1-\nu_i) k_i \frac{17}{6},$$
which simplifies to
$$|s_1| = O(N^{-\frac{1-\lambda_{\mathrm{ldn}}}{\beta-1}} + N^{-\lambda_{\mathrm{ldn}} \frac{\beta-1}{\beta-2}} w^{\beta-2}).$$
Lastly, the Berry-Esseen correction, similar to that in Theorem \ref{thm:star_ACS} Item (iii), reads
$$s_2 := \frac{C \int_0^1|\sin 2\pi x|^3\mathrm{d}x}{2^{-3/2} \sqrt{\nu_i k_i}} = O(N^{-\frac{1}{2\beta-2}}),~~~~\forall i\in\Hil_{\Delta}.$$

\newpage
\bibliographystyle{amsalpha}
\bibliography{references}
\end{document}